\documentclass{amsart}
\usepackage{tikz}\usetikzlibrary{arrows,positioning} 
\tikzset{
    >=stealth',
    punkt/.style={
           rectangle,
           rounded corners,
           draw=black, very thick,
           text width=6.5em,
           minimum height=2em,
           text centered},
    pil/.style={
           ->,
           thick,
           shorten <=2pt,
           shorten >=2pt,}
}

\usepackage{hyperref}
\usepackage[colorinlistoftodos,bordercolor=orange,backgroundcolor=orange!20,linecolor=orange,textsize=scriptsize]{todonotes}

 \usepackage[top=1in,bottom=1in,left=1.5in,right=1.5in]{geometry}
 \usepackage{amsmath,amssymb}
\usepackage{cleveref}

 \usepackage{latexsym}%
 \usepackage{pst-node}

\usepackage{setspace}%
\newcommand{\supp}[1]{\operatorname{supp}{#1}}
\newcommand{\support}{\mathcal{S}}

\newcommand{\differential}{D}
\newcommand{\diag}{\operatorname{diag}}
\newcommand{\digraph}{\vec{G}}
\newcommand{\diedges}{\vec{E}}
\newcommand{\kdigraph}{\vec{K}}
\newcommand{\Arcs}{\diedges}
 \newcommand{\C}{\mathbb{C}}
 \newcommand{\F}{\mathbb{F}}
 \newcommand{\N}{\mathbb{N}}
 
 \newcommand{\R}{\mathbb{R}}
 \newcommand{\Z}{\mathbf{Z}}

 \newcommand{\f}{\mathbf{f}}

 \newcommand{\m}{\mathbf{m}}
 \newcommand{\n}{\mathbf{n}}

 \newcommand{\p}{\mathbf{p}}
 \newcommand{\bF}{\mathbf{F}}
 \newcommand{\bG}{\mathbf{G}}

 \newcommand{\bt}{\mathbf{t}}

 \newcommand{\uu}{\mathbf{u}}
 \newcommand{\bv}{\mathbf{v}}

 \newcommand{\w}{\mathbf{w}}
 
 \newcommand{\x}{\mathbf{x}}
 
 \newcommand{\y}{\mathbf{y}}
 
 \newcommand{\z}{\mathbf{z}}
 \newcommand{\0}{\mathbf{0}}
 \newcommand{\1}{\mathbf{1}}

 \newcommand{\bj}{\mathbf{j}}

 \newcommand{\cA}{\mathcal{A}}
 \newcommand{\cB}{\mathcal{B}}
 \newcommand{\cC}{\mathcal{C}}
 
 \newcommand{\cE}{\mathcal{E}}
 \newcommand{\cF}{\mathcal{F}}
 \newcommand{\cG}{\mathcal{G}}
 
 \newcommand{\cI}{\mathcal{I}}
 \newcommand{\cJ}{\mathcal{J}}

 \newcommand{\cR}{\mathcal{R}}
 \newcommand{\cS}{\mathcal{S}}
 \newcommand{\cT}{\mathcal{T}}

 \newcommand{\lan}{\langle}
 \newcommand{\ran}{\rangle}
 \newcommand{\an}[1]{\lan#1\ran}
 \def\diag{\mathop{{\rm diag}}\nolimits}

 \newcommand{\Kron}{\mathop{\mathrm{Kr}}\nolimits}
 \newcommand{\trans}{^\top}

 \newcommand{\res}{\mathrm{res\;}}

 \newcommand{\pat}{\mathrm{pat}\;}

 \newcommand{\trop}{\mathrm{trop}}

 \newcommand{\rP}{\mathrm{P}}
 
 \newcommand{\rS}{\mathrm{S}}

 \newcommand{\rank}{\mathrm{rank\;}}

 \newtheorem{theo}{Theorem}%
 \newtheorem{theorem}[theo]{Theorem}

 \newtheorem{proposition}[theo]{Proposition}
 
 \newtheorem{lemma}[theo]{Lemma}

 \newtheorem{corollary}[theo]{Corollary}

\theoremstyle{remark}
\newtheorem{rem}{Remark}
\newtheorem{example}{Example}
 \numberwithin{equation}{section} %

 \renewcommand{\geq}{\geqslant}
 \renewcommand{\leq}{\leqslant}
 \renewcommand{\ge}{\geqslant}
 \renewcommand{\le}{\leqslant}
\begin{document}

 \title[Spectral inequalities for nonnegative tensors]{
   Spectral inequalities for nonnegative tensors and their tropical analogues}
 \author{S. Friedland}
\author{
 S. Gaubert%
 }
\date{%
}
\address{%
 Department of Mathematics, Statistics and Computer Science,
 University of Illinois at Chicago, Chicago, Illinois 60607-7045,
 USA, \texttt{friedlan@uic.edu},
 }
\address{%
INRIA and Centre de Math\'ematiques Appliqu\'ees (CMAP), \'Ecole polytechnique, UMR 7641 CNRS, 91128 Palaiseau C\'edex, \texttt{stephane.gaubert@inria.fr} }
\subjclass[2010]{%
 15A42, 15A69, 15A80.}

\keywords{%
Nonnegative tensors, spectral radius, tropical spectral radius, spectral norm, log-convexity, ergodic control, risk-sensitive control, entropy game.}

 \maketitle

 \begin{abstract}
We extend some characterizations and inequalities for the eigenvalues
of nonnegative matrices, such as Donsker-Varadhan, Friedland-Karlin,
Karlin-Ost inequalities,
to nonnegative tensors. Our approach involves
a correspondence between nonnegative tensors, ergodic control and entropy maximization: we show in particular that the logarithm of the spectral radius of a
tensor is given by en entropy maximization problem
over a space of occupation measures.
We study in particular the tropical analogue of the spectral radius,
that we characterize as a limit of the classical spectral radius,
and we give an explicit combinatorial formula for this tropical spectral
radius.
 \end{abstract}

 \section{Introduction}\label{sec:intro}
\subsection{Motivations}
 Nonnegative matrices are ubiquitous in mathematics, engineering, economics and computer science---
see our references.
For a square nonnegative matrix $A$, one of the most important concepts is the Perron-Frobenius eigenvalue (its spectral radius $\rho(A)$) and the corresponding eigenvector.
 For a rectangular matrix, a similar concept is the operator norm $\|A\|$ of $A$, which is given by the
 Perron-Frobenius norm of the induced symmetric matrix
 $S(A)=\left(\begin{smallmatrix}0&A\\A\trans &0\end{smallmatrix}\right)$.  In many applications, one uses a variational characterization of the Perron-Frobenius eigenvalue, the Collatz-Wielandt minimax formula.
The notion of irreducible matrix is also essential.  
Furthermore, the Perron-Frobenius eigenvalue $\rho(A)$ and the spectral
norm $\|A\|$ satisfy a number of convexity and logconvexity properties.
See for example \cite{Coh79,Fri81,Fribook,KO85,Kin61,nuss86}.
 
 In the last twenty years, there has been a tremendous interest and activity in tensors, which are multiarrays with at least $d\ge 3$ indices.
 Tensors come up in physics, in particular in quantum mechanics, and in various applications of engineering sciences, some of them being driven by data explosion.
See {{\cite{BGL17, DFLW17, FLS14, FL17, FT15, Lim05, Lim13, Lan12,ZNL18}}} and references therein.
 Since tensors do not represent linear operators, as matrices do, the theory of tensors is more delicate than the theory of matrices.  The spectral norm of tensors turns out to be one of the most important concept in theory and applications \cite{BGL17, HL13, FL17, Lim05}.  Unfortunately, it is generally NP-hard to compute the spectral norm \cite{HL13, FL17}---with exceptions like the case of symmetric {{qubits}} \cite{FW16}. 
Basic results for nonnegative matrices, including the existence and uniqueness of the Perron-eigenvalue, or the Collatz-Wielandt characterization of the Perron root, have been generalized to nonnegative tensors~\cite{CPZ08,FGH,Lim05,NQZ09}. An  inequality of Kingman (log-convexity of the spectral radius) has also been extended in~\cite{ZQL10}.
\subsection{Main results}
 In this paper, we extend several fundamental inequalities concerning nonnegative matrices, including the Donsker-Varadhan, Friedland-Karlin, Karlin-Ost, 
 inequalities,  to the case of nonnegative tensors. 
We also give generalizations to spectral norms of nonnegative tensors.
Our main results include a characterization of the logarithm of the spectral radius as the solution of an entropy maximization problem. They also include
an explicit combinatorial characterization of the tropical eigenvalue
of a nonnegative tensor, and a generalization of the classical bounds
of Cauchy, Birkhoff and Fujiwara for the largest
modulus of the root of a polynomial to the case of nonnegative tensors.
An ingredient here is an equivalence
between the Perron-Frobenius eigenproblem and an ergodic problem
arising in stochastic optimal control: we show that the
logarithm of the spectral radius of a nonnegative tensor coincides with the mean
payoff per time unit in a one player stochastic game problem, in which action spaces are simplices and payments are
given by a relative entropy. This is related 
to a work of Akian et al.,~\cite{akian_et_al:LIPIcs:2017:7026,AB17}
on the entropy game model of Asarin et al.,~\cite{asarin},
and to a work of Anantharanan and Borkar~\cite{AB17} on risk
sensitive control. The games considered in these approaches are associated
to families of nonnegative matrices. They 
differ from the present ones which are associated to tensors---except in the degenerate situation when these tensors are matrices (tensors with only $2$ indices). 
To our knowledge, this is the first work to connect
tensors and ergodic control.
As an application, we obtain a combinatorial
characterization of the tropical spectral radius.

The present results have an algorithmic impact. First,
the entropic characterization shows that the spectral
radius of a nonnegative tensor is the solution
of a {\em convex} optimization problem,
allowing a reduction to a widely studied class
of convex problems, see~\cite{CS16}.
Next, the correspondences that we develop allows one to apply to the Perron-Frobenius eigenproblem various algorithms developed in the setting of
ergodic control and zero-sum games. %
We note that connections between other algebraic problems
(linear and semidefinite feasibility problems) and well studied classes
of games (deterministic and stochastic mean payoff games) have been
previously developed in tropical geometry~\cite{AGG09,xasgmsissacjsc}.
It is of interest that nonnegative tensors correspond yet to another remarkable class of games.   {{A connection between tropical geometry and neural networks, discussed in a recent paper \cite{ZNL18}, also suggests that tropical tensors may be of interest in applications to data sciences.}}

 \subsection{Organization of the paper}
   We now survey the contents of this paper.  In \S\ref{sec:defirr} we recall the notions of indecomposability, for general nonnegative tensors,  and of irreducibility, for    nonnegative equidimensional  tensors. In \S\ref{sec:varspecrad} we discuss the spectral radius of a nonnegative equidimensional tensor and 
establish a formula for the first order
   perturbation of the spectral radius; interestingly, the proof is way more delicate than in the matrix case.  In \S\ref{sec:logconv}, we include
a generalization of Kingman's log-convexity theorem to the spectral radius of equidimensional weakly irreducible nonnegative tensors. In the case of irreducible tensors, this was first proved by Zhang, Qi, Luo and Xu in~\cite{ZQL10}. 
   In \S\ref{sec:FKineq} we generalize the inequality of Friedland-Karlin and the rescaling result in \cite{FK75}, and  the  Donsker-Varadhan characterization of the Perron-Frobenius eigenvalue~\cite{DV75} to equidimensional nonnegative tensors.  In \S\ref{sec:entropsr} we show
the logarithm of the spectral radius of  a nonnegative equidimensional tensor
coincides with the value of an ergodic control problem
of risk sensitive type.
    In \S\ref{sec:tropspecread} we discuss the tropical spectral radius of nonnegative tensors, denoted as $\rho_{\trop}(\cT)$, and the corresponding nonnegative eigenvector.   We give a combinatorial characterization of  $\rho_{\trop}(\cT)$ as a maximum of weighted cycle, which extends the characterization of the limit eigenvalue of Hadamard powers of a nonnegative matrices in \cite{Fri86}. 
This result reveals that $\log \rho_{\trop}(\cT)$ is given by the
value of an ergodic Markov decision process. 
We also give a generalization of Kingman's log-convexity theorem to $\rho_{\trop}(\cT)$.
   In \S\ref{sec:specnorm} we show that the spectral radius of a partially symmetric tensor is bounded by its spectral norm up to a combinatorial factor.
Thus, the spectral radius provides a tractable lowerbound of the spectral norm,
whereas, as noted above, computing the spectral norm is NP-hard.
We finally discuss generalizations of results \S\ref{sec:logconv} and \S\ref{sec:tropspecread} to the spectral norm
   of nonnegative tensors.

\begin{figure}
\begin{center}
\begin{tikzpicture}[node distance=1cm, auto]
\node (kingman) {log-convexity (Kingman), \cite{ZQL10}, Lemma~\ref{genkinin}
 } ;
 \node[inner sep=5pt,below=0.4cm of kingman]
 (FK) {Friedland-Karlin, Theorem~\ref{FKtens}};
 \node[above=of kingman] (dummy) {};
 \node[right=of dummy] (t) {Entropic characterization, Theorem~\ref{entropcharsrten}}
   edge[pil,bend left=45] (kingman.east); %

\node[below=0.4cm of FK] (KO) {Karlin-Ost, Theorem~\ref{genKO}};

\node[below=0.4cm of KO] (Tropical) {Friedland/tropical comparison, Corollary~\ref{cbf}} edge[pil,bend right=45,<-] (kingman.east);

\node[below=0.4cm of Tropical] (DV) {Donsker-Varadhan, Theorem~\ref{thm3.3F}};
\path[pil,draw,->] (KO) -- (Tropical);
\path[pil,draw,->] (kingman) -- (FK);
 \node[left=of dummy] (g) {Collatz-Wielandt, \cite{nuss86}, Eqn~(\ref{infmaxcharrho})}
   edge[pil, bend right=45] (kingman.west)
   edge[pil, bend right=45] (DV.west)
   edge[pil,<->, bend left=35] node[auto] {convex duality} (t);

\end{tikzpicture}
\end{center}
\caption{Relations between the different generalizations of matrix inequalities to nonnegative tensors. An arrow $A\to B$ indicates that inequality $A$ is used as a key ingredient in a known proof of inequality $B$.}
\end{figure}
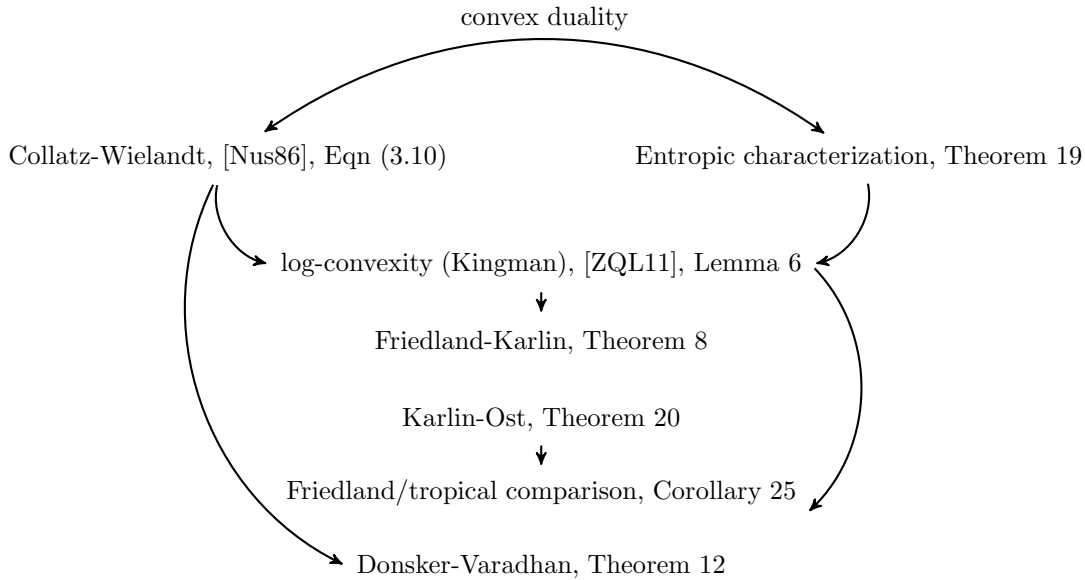
 \section{Definitions of irreducibility}\label{sec:defirr}
 In the case of square nonnegative matrices, irreducibility can be defined
 in two equivalent ways, either by requiring the directed graph
 associated with the matrix to be strongly connected,
 or by requiring that there is no non-trivial \emph{part}
 (relative interior of a face) of the standard positive cone
 that is invariant
 by the action of the matrix. Both requirements mean
 that the matrix cannot
 be put in upper block triangular form by applying the same permutation
 to its rows and columns. In the case of tensors, and more generally,
 of polynomial maps, the two approaches lead to distinct notions
 \cite{CPZ08,FGH}, as we next recall. 

Let $\F$ be either the field of complex numbers $\C$ or of real numbers $\R$.
 Denote
 \begin{eqnarray}\label{defbm}
 &&\m:=(m_1,\ldots,m_d),\quad m^{\times d}:=(\underbrace{m,\ldots,m}_{d\;\text{times}}),
\qquad [d]:=\{1,\dots,d\}\\
 &&\F^\m:=\otimes_{i\in[d]} \F^{m_i}=\F^{m_1\times\ldots\times m_d}, \qquad
 \F^{m^{\times d}}:= \otimes^d \F^m=\F^{m\times \ldots\times m},\notag
 \end{eqnarray}
 The vector space $\F^{m^{\times d}}$ is called the space of \emph{equidimensional} tensors.
 
We shall denote by $\R_+$ the set of nonnegative numbers.
 Then, $\R_+^\m\subset \R^\m$ is the cone of {\em nonnegative tensors}.  Assume that $\cF\in \R_+^\m$.
 We associate with  $\cF$ an {\em undirected} $d$-partite graph
 $G(\cF)=(V,E(\cF))$, the vertex set of which is the disjoint union
 $V=\cup_{j=1}^d V_j$, with $V_j=[m_j], j\in [d]$.
 The edge $(i_k,i_l)\in V_k\times V_l, k\ne l$ belongs
 to $E(\cF)$ if and only if $f_{i_1,i_2,\ldots,i_d}>0$ for some $d-2$
 indices $\{i_1,\ldots,i_d\}\backslash\{i_k,i_l\}$.
 The tensor $\cF$ is called \emph{weakly indecomposable} if the graph $G(\cF)$ is connected.

 We call $\cF$ \emph{indecomposable} if for each proper nonempty subset
 $\emptyset \ne I\subsetneqq V$,
 the following
 condition holds:  Assume that $I$ does not contain $V_p\cup V_q$ for any $p\ne q \in [d]$.
 Let $J:=V\backslash I$.  Then there exists
 $k\in [d]$, %
 $i_k\in I\cap V_k$ and $i_j\in J\cap V_j$ for each $j\in [d]
 \backslash\{k\}$ such that $f_{i_1,\ldots,i_d}>0$.
 It is shown in \cite{FGH} that if $\cF$ is indecomposable then $\cF$ is
 weakly indecomposable.

 Assume that $\cF$ is an equidimensional tensor in $\R_+^{m^{\times d}}$.
 With $\cF$ we associate a directed graph $\digraph(\cF)=(V, E(\cF))$, where $V=[m]$.  The diedge from $i$ to $j$ belongs to $E(\cF)$
 if and only if $f_{i,j_1,\ldots,j_{d-1}}>0$ for some $d-1$ indices $\{j_1,\ldots,j_{d-1}\}$ such that $j=j_k$ for some $k\in[d-1]$.
 We say that $\cF$ is {\em weakly irreducible} if $\digraph(\cF)$ is strongly connected.

 We call $\cF$ \emph{irreducible} if for each proper nonempty subset
 $\emptyset \ne I\subsetneqq V$,
 there exist $i\in I$ and %
$j_1,\ldots,j_{d-1}\in V\setminus I$
 such that $f_{i,j_2,\ldots,j_{d-1}}>0$.
 Our definition of irreducibility agrees with \cite{Lim05,CPZ08,NQZ09}.
 The following lemma follows from the results in \cite{FGH}. 
 \begin{lemma}\label{strngimpweak}  Let $\cF\in \R_+^\m$.  Then
 \begin{enumerate}
 \item If $\cF$ is indecomposable then $\cF$ is weakly indecomposable.
 \item Assume that $m_1=\ldots=m_d=m$.  If $\cF$ is irreducible then $\cF$ is weakly irreducible.
 \end{enumerate}
 \end{lemma}

 In the paper \cite{FGH}, the notions of weak indecomposability and indecomposability were called weak irreducibility
 and irreducibility.  To avoid the ambiguity, we used here two different terms: indecomposability of general
 tensors, in the context of multilinear forms, and irreduciblity of equidimensional tensors, in the context of polynomial maps.

 \section{The spectral radius of an equidimensional tensor}\label{sec:varspecrad}
 \subsection{Standard facts on tensors}
 Let $\m=(m_1,\ldots,m_d), \n=(n_1,\ldots,n_e)$.
  Assume that $\cF=[f_{i_1,\ldots,i_d}]\in\F^{\m},\cG=[g_{j_1,\ldots,j_e}]\in \F^{\n}$ are given.  Then the entries of the tensor product $\cF\otimes\cG\in\F^{(\m,\n)}:=\F^{\m}\otimes \F^{\n}$ are given by $ [f_{i_1,\ldots,i_d}g_{j_1,\ldots,j_e}]$.
  
Assume that $d=e$.  Then $\cG$ is called a subtensor of $\cF$ if the following conditions hold. First, $n_k\le m_k$ for $k\in [d]$.  Second,   for each $k\in[d]$ there exists a sequence 
$1\leq i_{1,k}<\ldots<i_{n_k,k}\leq m_k$ such that $g_{j_1,\ldots,j_d}=f_{i_{j_1,1},\ldots,i_{j_d,d}}$.

Define $\m\circ\n:=(m_1n_1,\ldots,m_dn_d)$.
 Then $\F^{(\m,\n)}$ and $\F^{\m\circ\n}$ are isomorphic as vector spaces.  Furthermore, the isomorphism $\iota:\F^{(\m,\n)}\to\F^{\m\circ\n}$
 maps rank one tensors to rank one tensors, but $\iota^{-1}$ does not preserves the rank one tensors.  
We define the Kronecker product of tensors $\cF\otimes_{\Kron}\cG=[h_{(i_1,j_1),\ldots,(i_d,j_d)}]\in \F^{\m\circ\n}$, where $h_{(i_1,j_1),\ldots,(i_d,j_d)}=f_{i_1,\ldots,i_d}g_{j_1,\ldots,j_d}$, which extends the classical definition of the Kronecker product of matrices.

 Let $J=\{1\le j_1<\ldots<j_k\le d\}$ be a nonempty subset of $[d]$.  Denote $\m(J)=(m_{j_1},\ldots,m_{j_k})$.
 Assume that $\cT=[t_{i_1,\ldots,i_d}]\in\F^\m,\cS=[s_{i_{j_1},\ldots,i_{j_k}}]\in\F^{\m(J)}$.
 Then $\cT\times \cS=\cS\times \cT=\F^{\m([d]\setminus J)}$ is a $d-k$ tensor obtained by the contraction on the indices
 in $J$.  That is, the entries of $\cT\times \cS$ are
 \[\sum_{i_{j_1}\in [m_{j_1}],\ldots, i_{j_k}\in[m_{j_k}]} t_{i_1,\ldots,i_d} s_{i_{j_1} ,\ldots,i_{j_k}}.\]
 Furthermore, we define the Hadamard product $\cT\circ \cS=\cS\circ\cT:=[t_{i_1,\ldots,i_d} s_{i_{j_1},\ldots,i_{j_k}}]\in \F^\m$.
 Assume that $J=[d]$, i.e.\ $\m=\n$.  Then $\cT\circ \cS\in\F^{\m}$ can be viewed as a \emph{subtensor} of $\cT\otimes_{\Kron}\cS$
 where we choose $i_1=j_1\in [m_1],\ldots,i_d=j_d\in [m_d]$.
 Observe next that $\cT\times\cS$ is a scalar.  In fact, $\an{\cT,\cS}:=\cT\times \cS$ is an inner product on $\R^\m$.
The Hilbert-Schmidt norm on $\R^\m$ is defined by
 $\|\cT\|:=\sqrt{\an{\cT,\cT}}$.  The Cauchy-Schwarz inequality yields that $|\cT\times \cS|\le \|\cT\| \|\cS\|$.
 For $\x_j\in \F^{m_j}, j\in J$, we denote $\otimes_{j\in J} \x_j:=\x_{j_1}\otimes\dots\otimes \x_{j_k}\in \F^{\m(J)}$.
 Furthermore for $\x=(x_1,\ldots,x_m)\trans\in \F^m$ and for any positive integer $k$ we denote $\otimes^k {\x}:=\underbrace{{\x}\otimes\dots\otimes{\x}}_{k\;\text{times}}\in \F^{m^{\times k}}, \x^{\circ k}:=(x_1^k,\ldots,x_m^k)\trans\in \F^m$.
\subsection{Symmetric tensors}\label{subsec:symten}
 In this subsection we recall the standard facts about symmetric tensors we use in this paper.  We follow closely \cite{FW16}.  A tensor $\cS=[s_{i_1,\ldots,i_d}]\in\F^{n^{\times d}}$ is called symmetric  if $s_{i_1,\ldots,i_d}=
s_{i_{\omega(1)},\ldots,i_{\omega(d)}}$ for every permutation $\omega:[d]\to[d]$.  
Denote by $\rS^d\F^n\subset \otimes^d\F^{n^{\times d}}$ the vector space of $d$-mode symmetric tensors on $\F^n$.  A tensor
$\cS\in\rS^d\F^n$ defines a unique homogeneous polynomial of degree $d$ in $n$ variables $f(\x)=\cS\times\otimes^d\x$, where
\begin{eqnarray*}\label{defpolfx}
f(\x)=\sum_{ j_k+1\in [d+1],k\in[n], j_1+\cdots +j_n=d} \frac{d!}{j_1!\cdots j_n!} f_{j_1,\ldots,j_n} x_1^{j_1}\cdots x_n^{j_n}.
\end{eqnarray*}
Conversely, a homogeneous polynomial $f(\x)$ of degree $d$ in $n$ variables defines a unique symmetric $\cS\in\rS^d\F^n$ as shown below.
 Denote by $J(d,n)$ be the set of all $\bj=(j_1,\ldots,j_n)\in\Z_+^n$ appearing in the above definition of $f(\x)$:
\begin{eqnarray*}\label{defJdn} 
J(d,n)= \{\bj=(j_1,\ldots,j_n)\in\Z_+^n,\; j_1+\cdots+j_n=d\}.
\end{eqnarray*}
It is well known that $|J(d,n)|={n+d-1\choose n}={n+d-1\choose d-1}$ \cite{FW16}.
Define $c(\bj)=\frac{d!}{j_1!\cdots j_n!}$.
For $\x=(x_1,\ldots,x_n)\trans\in\F^n$ and $\bj=(j_1,\ldots,j_n)\in J(d,n)$ let $\x^{\bj}$ be the monomial $x_1^{j_1}\cdots x_n^{j_n}$.  Then the above definition of $f(\x)$  is equivalent to $f(\x)=\sum_{\bj\in J(d,n)} c(\bj)f_{\bj}\x^{\bj}$.

Observe next that $\F^{J(d,n)}$ is isomorphic to $\rS^d\F$, where $\cS=[s_{i_1.\ldots,i_d}]\in \rS^d\F^n$ corresponds to $\f=(f_{\bj})$ as follows:  For $i_j\in[n],j\in [d]$ and $k\in[n]$ let $j_k$ be the number of times $k$ appears in the multiset $\{i_1,\ldots,i_d\}$.  Then $f_{(j_1,\ldots,j_n)}=s_{i_1.\ldots,i_d}$.  Thus $\rS^d\F^n$ is isomorphic to $\rP(d,n,\F)$, the space of homogeneous  polynomials of degree $d$ in $n$ variables.  In particular, the dimension of $\rS^d\F^n$ is $n+d-1\choose d$ which is much smaller then $n^d$, the dimension of $\F^{n^{\times d}}$.  (Note that dim $\rS^d\F^n$ is $O(d^n)$ for a fixed value of $n$.)  

Assume that $\cS\in\rS^d\F^n$ and let $f(\x)=\cS\times\otimes^d\x$.  Then \cite{FW16}:
\[\cS\times(\otimes^{d-1}\x)=\frac{1}{d}\nabla f(\x)=\frac{1}{d}(\frac{\partial f}{\partial x_1}(\x),\ldots,\frac{\partial f}{\partial x_n}(\x)).\]

\subsection{The homogeneous eigenvalue problem}\label{subsec:homeigv}
 With an equidimensional tensor $\cF\in \C^{n^{\times d}}$ we associate a homogeneous map of degree $d-1$
 given as $\x\mapsto\cF(\x)=\cF\times \otimes^{d-1}\x$, where the contraction is \emph{on the last}
 $d-1$ indices of $\cF$.  Hence, without loss of generality, we may assume that $\cF=[f_{i_1,\ldots,i_d}]$ is symmetric with respect to
the indices $i_2,\ldots,i_d$.  For $\F\in\{\C,\R\}$ we denote by $\F^{n^{\times d}}_{ps}$ the subspace of tensors whose entries $t_{i_1,\ldots,i_d}$ are symmetric with respect to the indices $i_2,\ldots,i_d$.  We call such tensors \emph{partially symmetric}.  
Denote by $\R_{ps,+}^{n^{\times d}}\subset \R_{ps}^{n^{\times d}}$ the cone of  nonnegative partially symmetric tensors.

We observe that a partially symmetric tensor $\cF\in\F^{n^{\times d}}_{ps}$ is in one-to-one correspondence with a polynomial map $\bF=(F_1,\ldots,F_n):\F^n\to\F^n$, where each $F_i(\x)=(\cF\otimes (\otimes^{d-1}\x))_i$ is a homogeneous polynomial of degree $d-1$ in $\x$.  Hence $\F^{n^{\times d}}_{ps}$ is isomorphic to $\rP(d-1,n,\F)^d$, and its dimension is $n{n+d-2\choose d-1}$.

Assume that $\cF\in\R_{ps,+}^{n^{\times d}}$.  Then the digraph  $\digraph(\cF)$ contains a diedge $(i,j)$ if and only if $\frac{\partial F_i(\x)}{\partial x_j}> 0$ for $\x>\0$. 

The homogeneous eigenvalue problem considered in this paper is
 \begin{equation}\label{homeigprob}
 \cF\times \otimes^{d-1}\x=\lambda\x^{\circ (d-1)}, \quad \x\ne \0.
 \end{equation}
 For $d=2$, i.e.\ when $\cF$ is a square matrix, the above homogeneous eigenvalue problem is the standard eigenvalue problem for matrices.
 We restrict our attention to $d>2$. As for matrices, for $\cF,\cS\in \C^{n^{\times d}}_{ps}$ we can consider the pencil eigenvalue problem
 \begin{equation}\label{homepeneigprob}
 \cF\times \otimes^{d-1} \x=\lambda \cS\times \otimes^{d-1} \x, \quad \x\ne \0.
 \end{equation}
 For $\cS=\cI_{n,d}$, where $\cI_{n,d}$ is the diagonal tensor $[\delta_{i_1i_2}\ldots\delta_{i_1i_d}]\in \C^{n^{\times d}}_{ps}$, the system
 \eqref{homepeneigprob} reduces to \eqref{homeigprob}.  When no ambiguity arises we denote $\cI_{n,d}$ by $\cI$.

The tensor $\cS$ is called \emph{singular} if the system
 \begin{equation}\label{homSeq}
 \cS\times\otimes^{d-1}\x=\0
 \end{equation}
 has a nontrivial solution.  Otherwise $\cS$ is called nonsingular.  Recall the classical notion on the resultant corresponding
 to the system \eqref{homSeq}.  There exists an irreducible polynomial $\res :\C^{n^{\times d}}_{ps}\to \C$ with the following properties
 \cite[Chapter 13]{GKZ}.  First, $\cS$ is singular if and only if $\res \cS=0$.  Second for a general singular $\cS$ the set of all nontrivial
 solutions is a line, i.e., a one dimensional vector space.  Third, the degree of $\res $ is $n (d-1)^{n-1}$.
 Hence, to find all eigenvectors of the system \eqref{homepeneigprob},
one first finds all the solutions of the \emph{characteristic equation}
 \begin{equation}\label{chareqpen}
 \res(\lambda\cS-\cF)=(\res \cS)\lambda^{n(d-1)^{n-1}}+\sum_{j\in [n(d-1)^{n-1}]} c_j(\cS,\cF)\lambda^{n(d-1)^{n-1}-j}.
 \end{equation}
 Here $c_j(\cS,\cF)$ is a homogeneous polynomial of total degree $n(d-1)^{n-1}$ and the 
partial degrees
in the $\cS$ and $\cF$ variables are $n(d-1)^{n-1}-j$ and $j$
 respectively.  After finding all the solutions of \eqref{chareqpen},
called the {\em eigenvalues}
 of the pencil $(\cF,\cS)$, one needs to find the corresponding
 eigenvectors.  If $\cS$ is nonsingular then the pencil $(\cF,\cS)$ has exactly $n(d-1)^{n-1}$ eigenvalues counting with multiplicities.

 We now restrict our attention to the homogeneous eigenvalue problem \eqref{homepeneigprob}.
 Clearly, $\cI$ is a nonsingular tensor.  This case is studied in \cite[\S5]{FO12}.  Let $\lambda_1(\cF)$, \ldots, $\lambda_{n(d-1)^{n-1}}(\cF)$
 be the solutions of the characteristic equation~\eqref{chareqpen} corresponding to $\cS=\cI$.  Then a general $\cF$ has $n(d-1)^{n-1}$ distinct
 eigenvalues, and to each eigenvalue $\lambda_i$ corresponds a unique eigenvector $\x_i\ne \0$, up to a nonzero factor.  (I.e., the eigenspace is the line
in $\C^n$ spanned by $\x_i$.)
 Let
 \begin{equation}\label{defspecradF}
 \rho(\cF):=\max\{|\lambda_i(\cF)|, \;i\in [n(d-1)^{n-1}]\}
 \end{equation}
 be the \emph{spectral radius of} $\cF$.  
Since the roots of a polynomial depend continuously of its coefficients,
using the characteristic equation \eqref{chareqpen},
we arrive at
the following result.
 \begin{proposition}\label{contspecrad}  Let $\cF\in \C^{n^{\times d}}_{ps}$.  Let $\rho(\cF)$ be the spectral radius for the eigenvalue problem \eqref{homeigprob}
 given by \eqref{defspecradF}.  Then $\rho(\cF)$ is a continuous function on $\C^{n^{\times d}}_{ps}$.\hfill\qed
 \end{proposition}

 Let $\cE=[e_{j_1,\ldots,j_d}]\in \C^{m^{\times d}}_{ps},\cF\in \C^{n^{\times d}}_{ps}$.  Assume that $\cE$ has a homogeneous eigenvector
 \begin{equation}\label{Ehomeig}
 \cE\times\otimes^{d-1}\y=\mu \y^{\circ(d-1)}, \quad \y\neq \0.
 \end{equation}
 Assume that $\x$ is a homogeneous eigenvector of $\cF$, as in~\eqref{homeigprob}.  Then a straightforward computation shows that
 \begin{equation}\label{EFKronprdeig}
 (\cE\otimes_{\Kron}\cF)\times \otimes^{d-1}(\y\otimes\x)=\mu\lambda (\y\otimes\x)^{\circ(d-1)}.
 \end{equation}
 Hence we deduce the inequality
 \begin{equation}\label{EFKronspecradin}
 \rho(\cE)\rho(\cF)\le \rho(\cE\otimes_{\Kron}\cF).
 \end{equation}
 For matrices, i.e.\ when $d=2$, the equality holds.
This follows from the fact that
 the number of eigenvalues of $\cE\otimes_{\Kron}\cF$ of the form $\mu\lambda$ is exactly $mn$, which is the total number
 of the eigenvalues of the matrix $\cE\otimes_{\Kron}\cF$. For $d>2$ the number of eigenvalues of the form $\mu\lambda$
 is $(m(d-1)^{m-1})(n(d-1)^{n-1})$ which is strictly less than $(mn)(d-1)^{(mn)-1}$, the number of the eigenvalues of $\cE\otimes_{\Kron}\cF$,
 for $m,n>1$.  So it is not clear that the equality in \eqref{EFKronspecradin} always holds for $d>2$.
 We will show in the next subsection that for nonnegative tensors, the equality does hold in~\eqref{EFKronspecradin}.

 \subsection{Spectral radius of nonnegative tensors}
 Let $\cT=[t_{i_1,\ldots,i_d}]\in \R^{n^{\times d}}_{ps}$. 
 We now summarize the known results on the spectral radius of $\cT\in\R_{ps,+}^{n^{\times d}}$ which will be used here, see~\cite{CPZ08,FGH}. 
Some of these results carry over to non-linear order preserving positively homogeneous self-maps of the standard orthant, see~\cite{nuss86,GG04}.
 \begin{theorem}\label{propsecradten}  Let $\cT\in\R_{ps,+}^{n^{\times d}}$.  Then $\rho(\cT)$ is an eigenvalue of $\cT$ corresponding to a nonnegative
 eigenvector
 \begin{equation}\label{noneigwkir}
 \cT(\bv)=\rho(\cT)\bv^{\circ(d-1)}, \quad \bv \gneq \0.
 \end{equation}
 Furthermore
 \begin{equation}\label{infmaxcharrho}
 \rho(\cT)=\inf_{\x=(x_1,\ldots,x_n)\trans>\0} \max_{i\in [n]}\frac{\cT(\x)_i}{x_i^{d-1}}.
 \end{equation}
 Assume that $\cT$ is irreducible.  Then $\cT$ has a nonnegative eigenvector $\uu$, which is positive, and unique (up to a scalar multiple).
The corresponding eigenvalue is the spectral radius $\rho(\cT)$
 \begin{equation}\label{poseigvec}
 \cT(\uu)=\rho(\cT)\uu^{\circ(d-1)}, \quad \uu>\0.
 \end{equation}
 Furthermore $\rho(\cT)$ has the characterizations
 \begin{equation}\label{minmaxrhoirchar}
 \rho(\cT)=\min_{\x>\0} \max_{i\in [n]}\frac{\cT(\x)_i}{x_i^{d-1}}=\max_{\x>\0} \min_{i\in [n]}\frac{\cT(\x)_i}{x_i^{d-1}}.
 \end{equation}
 Assume that $\cT$ is weakly irreducible.  Then $\cT$ has a unique positive eigenvector $\uu$,
which satisfies \eqref{poseigvec}.
 Furthermore
 \begin{equation}\label{minmaxrhowirchar}
 \rho(\cT)=\min_{\x>\0} \max_{i\in [n]}\frac{\cT(\x)_i}{x_i^{d-1}}=\max_{\x\gneq\0} \min_{i\in [n], x_i>0}\frac{\cT(\x)_i}{x_i^{d-1}}.
 \end{equation}
 \end{theorem}
We next show how these properties can be derived from known results. In particular, the variational characterizations of the spectral radius in~\eqref{infmaxcharrho}, \eqref{minmaxrhoirchar}, \eqref{minmaxrhowirchar} follow from a general Collatz-Wielandt formul\ae\ of Nussbaum for nonlinear maps. 
\begin{proof}
 Suppose first that $\cT$ is irreducible.  Then the results in \cite{CPZ08} show that any nonnegative eigenvector is positive, and that this eigenvector
is unique up to a scalar factor. 
 It corresponds to a positive eigenvalue which is the spectral radius of $\cT$.  Furthermore, the characterization in \eqref{minmaxrhoirchar} holds.

 Assume that $\cT\ge 0$ is not irreducible.  First, we shall use
a perturbation argument to deduce that $\rho(\cT)$ is a eigenvalue of $\cT$ corresponding
 to a nonnegative eigenvector satisfying \eqref{noneigwkir}.  Let $\cJ_{n,d}\in\R_{ps,+}^{n^{\times d}}$ be a tensor all the entries of which are $1$.  Assume that $\epsilon>0$.  Then $\cT+\epsilon \cJ_{n,d}>0$.  Hence there exists a positive probability vector $\uu(\epsilon)$ so that
 \[(\cT+\epsilon\cJ_{n,d})(\uu(\epsilon))=\rho(\cT+\epsilon \cJ_{n,d})\uu(\epsilon)^{\circ(d-1)}.\]
 From the first characterization \eqref{minmaxrhoirchar} we deduce that $\rho(\cT+\epsilon\cJ_{n,d})$ is a nondecreasing function on $(0,\infty)$.
 Since $\rho(\cS)$ is a continuous function in $\cS\in \C^{n^{\times d}}_{sp}$ it follows that $\lim_{\epsilon\searrow 0}\rho(\cT+\epsilon\cJ_{n,d})=\rho(\cT)$.
 Observe next that there exists a decreasing sequence $\epsilon_j>0, j\in\N$ converging to zero such that $\uu(\epsilon_j)$ converge to a probability vector $\bv=(v_1,\ldots,v_n)$.    Since $\uu(\epsilon_j)=(u_{1,j},\ldots,u_{n,j})\trans$ is an eigenvector of $\cT+\epsilon_j\cJ_{n,d}$ corresponding to $\rho(\cT+\epsilon_j\cJ_{n,d})$ we deduce \eqref{noneigwkir}.
 
 The results in \cite[\S3]{nuss86} yield that for any nonnegative tensor $\cT$ with maximal nonnegative eigenvalue $\rho(\cT)$, the characterization \eqref{infmaxcharrho}
 holds.  (To apply the results in \cite{nuss86} we need to consider the homogeneous map of degree one $\uu\mapsto (\cT(\uu))^{\circ\frac{1}{d-1}}$ for $\uu\ge \0$.
 Here for $\uu=(u_1,\ldots,u_n)\trans\ge \0$ we denote by $\uu^{\circ t}$ the vector $(u_1^t,\ldots,u_n^t)\trans$ for $t>0$.  See for more details \cite{FGH}.)
 One can also use the above perturbation technique to deduce 
\eqref{infmaxcharrho}.

The statements of the theorem for a weakly irreducible tensor $\cT$ follow
from~\cite[Corollary 4.2]{FGH}.
\end{proof}

 We now give the first variation of the eigenvalue $\lambda=\rho(\cT)$ for a weakly irreducible tensor $\cT\in\R_+^{n^{\times d}}$.
We denote by $\uu=(u_1,\ldots,u_n)\trans>\0$
the corresponding positive eigenvector.
Note that we can assume without loss of generality
 that $u_n=1$.  
We suppose that $\cR\in \C^{n^{\times d}}_{ps} $ is a partially symmetric tensor in the neighborhood of $\cT$, and we
are interested in the spectral radius $\lambda$ of this tensor.
Thus, we have a system of $n$ nonlinear equations in $n$ unknowns, consisting of $z_1,\ldots,z_{n-1}$, the entries of $\z=(z_1,\ldots,z_{n-1},1)\trans\in\C^n$  and of the eigenvalue $\lambda$,
given by  \begin{equation}\label{eigeq}
 \bG(\z,\lambda,\cR)=\0, \quad \bG(\z,\lambda,\cR):=\cR(\z)-\lambda\z^{\circ(d-1)}
\enspace .
 \end{equation}
 We look for a solution
$(\z,\lambda)$ in the neighborhood of $(\uu,\rho(\cT))$.
We shall apply the implicit function theorem 
after showing that the Jacobian of $\bG$ with respect to $(\z,\lambda)$ at $(\uu,\rho(\cT),\cT)$
 has rank $n$.

 First observe that
 \begin{equation}\label{partderT}
 \cT(\x+\y)=\cT(\x)+\differential \cT(\x)\y + O(\|\y\|^2), \quad \differential \cT(\x):=(d-1)\cT\times \otimes^{d-2}\x \in \R^{n\times n},
 \end{equation}
where $\differential \cT(\x)$ denotes the differential map of $\cT$ at point
$\x$.
 In the last expression the contraction is on the last $d-2$ indices of $\cT$.
 Second, assume that $\cT\in\R_{ps,+}^{n^{\times d}}$ is weakly irreducible.  Assume \eqref{poseigvec} holds.  Then $\differential\cT(\uu)\in \R_+^{n\times n}$
 is an irreducible matrix satisfying
 \begin{equation}\label{pratTuid}
 \differential\cT(\uu)\uu=(d-1)\rho(\cT)\uu^{\circ(d-1)}.
 \end{equation}
 For a vector $\x=(x_1,\ldots,x_n)\in \R^n$ denote by $\diag(\x)\in \R^{n\times n}$  the diagonal matrix $\diag(x_1,\ldots,x_n)$. Set
 \begin{equation}\label{defATu} 
 A:=\diag(\uu)^{-(d-2)}\differential\cT(\uu), \quad\textrm{where }\cT(\uu)=\rho(T)\uu^{\circ (d-1)}, \uu>\0.
 \end{equation}
 Since $\differential\cT(\uu)$ is irreducible and $\uu>0$, it follows that $A$ is an irreducible matrix.  Furthermore, there exists a unique vector $\0<\w\in\R^n$ such that the following conditions hold
 \begin{equation}\label{defpropA}
 A\uu=(d-1)\rho(\cT)\uu, \quad A\trans\w=(d-1)\rho(\cT)\w,\quad\w\trans \uu=1.
 \end{equation}
 \begin{theorem}\label{pertspecrad}  Let $\cT=[t_{i_1,\ldots,i_d}]\in\R_{ps,+}^{n^{\times d}}$ be weakly irreducible.  
\begin{enumerate}
\item Assume that \eqref{poseigvec} holds. Then,
there exists analytic functions
 $\z(\cR)$ and $\lambda(\cR)$ in the $n^d$ entries of $\cR\in\C^{n^{\times d}}_{ps}$, defined in the neighborhood of $\cT$,
satisfying $\z(\cT)=\uu$ and $\lambda(\cT)=\rho(\cT)$.
\item
Furthermore, let $\cS=[s_{i_1,\ldots,i_d}]\in \R^{n^{\times d}}_{ps}$ be such that $s_{i_1,\ldots,i_d}\ge 0$ if $t_{i_1,\ldots,i_d}=0$.  Then, for a small $\epsilon\ge 0$, one has the following expansion
 \begin{equation}\label{rhoTpert}
 \rho(\cT+\epsilon\cS)=\rho(\cT)+\epsilon \w\trans \diag(\uu)^{-(d-2)}\cS(\uu) +O(\epsilon^2),
 \end{equation}
 where $\w$ is the positive vector defined in \eqref{defpropA}.
\end{enumerate}
 \end{theorem}
 \begin{proof} 
 Let $\bG(\z,\lambda,\cR)$ be defined as in \eqref{eigeq}, where $\z=(z_1,\ldots,z_{n-1},1)\trans$.  We next show that $\differential_{\z,\lambda}\bG(\uu,\rho(\cT),\cT)$,  i.e.\ the Jacobian of $\bG$ with respect to $(\z, \lambda)$ at the point $(\uu,\rho(\cT),\cT)$, has rank $n$.  

The derivative of $\cT(\z)-\lambda\z^{\circ(d-1)}$ with respect to the variable $z_i$ gives the $i$-th column of the matrix
 $\differential\cT(\z)-(d-1)\diag(\z)^{d-2}$ for $i=1,\ldots,n-1$.  The derivative of $\cT(\z)-\lambda\z^{\circ(d-1)}$ with respect to $\lambda$
 gives the column $-\z^{\circ{d-1}}$.  So the matrix $B:=\differential_{\z,\lambda}\bG(\uu,\rho(\cT),\cT)\in\R_+^{n\times n}$ is given as follows.
 Its first $n-1$ columns are the first $n-1$ columns of $\differential\cT(\uu)-(d-1)\rho(\cT)\diag(\uu)^{d-2}$.  The last column of $B$ is $-\uu^{d-1}$.
 Let $C:=\diag(\uu)^{-(d-2)}B$.  Then the first $n-1$ columns of $C$ are the first $n-1$ columns of $A-(d-1)\rho(\cT)I$, where $A$ is defined in
 \eqref{defATu}.  The last column of $C$ is $-\uu$.   The Perron-Frobenius theorem
 yields that $\rho(A)=(d-1)\rho(\cT)$. Moreover, since $A$ is irreducible, the eigenspace of $A$
associated to the spectral radius of $A$ is of dimension $1$.
As $(A-(d-1)\rho(\cT)I)\uu=0$, we deduce that
 the unique linear combination of the columns of $A-(d-1)\rho(\cT)I$,  up to a nonzero scalar,  which is a zero vector, is given by the coordinates of $\uu$.
Since $\uu_{n}\neq 0$, it follows that
the first $n-1$ columns of $A-(d-1)\rho(\cT)I$
 are linearly independent.  From the definition of $\w>\0$ in \eqref{defpropA} it follows that the first $n-1$ columns of $A-(d-1)\rho(\cT)I$
 form a basis to the subspace of $\R^n$ orthogonal to the vector $\w$.  By the definition $\w\trans \uu=1$.  Hence $\uu$ is not a linear combination
 of the first $n-1$ columns of $A-(d-1)\rho(\cT)I$.  So the columns of $C$ are linearly independent, i.e.\ $\rank C=n$.  Therefore $\rank B=n$.
 Since $\bG(\z,\lambda,\cR)$ is analytic in $(\z,\lambda,\cR)$ the implicit function theorem implies that there exists analytic functions
 $\z(\cR),\lambda(\cR)$ in the $n{n+d-2\choose d-1}$ entries of $\cR\in\C^{n^{\times d}}_{ps}$ in the neighborhood of $\cT$ satisfying $\z(\cT)=\uu,\lambda(\cT)=\rho(\cT)$.

 Let $\cS=[s_{i_1,\ldots,i_d}]\in \R^{n^{\times d}}_{ps}$ such that $s_{i_1,\ldots,i_d}\ge 0$ if $t_{i_1,\ldots,i_d}=0$.
 Assume $a>0$ satisfies the condition that $t_{i_1,\ldots,i_d}+as_{i_1,\ldots,i_d}>0$ if $t_{i_1,\ldots,i_d}>0$.  Then $\cT+\epsilon\cS\in\R_{ps,+}^{n^{\times d}}$
 is weakly irreducible for $\epsilon\in[0,a]$.  Thus, the functions
$\lambda(\cT+\epsilon \cS)$ and $\z(\cT+\epsilon\cS)$ of the parameter
$\epsilon$ are analytic in some small open disc $|\epsilon|<r\le a$,
 and $\lambda(\cT+\epsilon \cS)=\rho(\cT+\epsilon\cS)$ for $\epsilon \in [0,r)$.
For $\epsilon\in [0,r)$,
one has the following expansion
 \[\rho(\cT+\epsilon\cS)=\rho(\cT)+\mu\epsilon +O(\epsilon^2), \quad \z(\cT+\epsilon\cS)=\uu+\epsilon\y+O(\epsilon^2).\]
 Inserting these expressions in the equality $\bG(\z(\cT+\epsilon\cS),\rho(\cT+\epsilon\cS), \cT+\epsilon\cS)=\0$ we 
must have that the coefficient of $\epsilon$
 is zero.  This is equivalent to the equality
 \[\differential\cT(\uu)\y+\cS(\uu)-(d-1)\rho(\cT)\diag(\uu)^{d-2}\y-\mu\uu^{\circ(d-1)}=\0.\]
We multiply the above equality by $\diag(\uu)^{-(d-2)}$, and rearrange the terms to deduce the equality
 \[(A-(d-1)\rho(\cT)I)\y+\diag(\uu)^{-(d-2)}\cS(\uu)-\mu\uu=\0.\]
We now multiply from the left by the vector $\w\trans$.  We finally use \eqref{defpropA} to deduce $\mu=\w\trans \diag(\uu)^{-(d-2)}\cS(\uu)$.
 This establishes \eqref{rhoTpert}.  \end{proof}

The following proposition is well known for matrices, and its extension to tensors is also known. 
 \begin{proposition}\label{specradprodeq}  Let $\cE\in\R_{ps,+}^{m^{\times d}},\cF\in\R_{ps,+}^{n^{\times d}}$.  Then $\rho(\cE)\rho(\cF)=\rho(\cE\otimes_{\Kron}\cF)$.
 Suppose furthermore that $m=n$. Then
 \begin{equation}\label{specradcircin}
 \rho(\cE\circ\cF)\le \rho(\cE)\rho(\cF)
 \end{equation}
 \end{proposition}
 \begin{proof} Assume that $\cE>0,\cF>0$.  Let $\y>\0,\x>\0$ be the positive eigenvectors corresponding to the eigenvalues $\rho(\cE),\rho(\cF)$ respectively.
 Then $\rho(\cE)\rho(\cF)$ is a positive eigenvalue of $\cE\otimes_{\Kron}\cF$ corresponding to the positive eigenvalue $\y\otimes\x$.  The results of \cite{CPZ08}
 yield the equality $\rho(\cE)\rho(\cF)=\rho(\cE\otimes_{\Kron}\cF)$.  Clearly
 Characterization \eqref{infmaxcharrho} yields the inequality \eqref{specradcircin}.  The results for nonnegative $\cE,\cF$ is derived using the continuity
 argument as in the proof of Theorem \ref{propsecradten}.  \end{proof}

 \section{Logconvexity of the spectral radius of nonnegative tensors}\label{sec:logconv}
  Given a tensor $\cA=[a_{i_1,\ldots,i_d}]\in \R_{pr,+}^{\m}$ and a real nonnegative number $p$, we set $\cA^{\circ p}:=[a_{i_1,\ldots,i_d}^p]\in \R_{ps,+}^{\m}$.
 (Here $0^0=0$ unless stated otherwise.)
 \begin{lemma}\label{genkinin}
 Let $\cF=[f_{i_1,\ldots,i_d}],\cG=[g_{i_1,\ldots,i_d}]\in\R_{pr,+}^{n^{\times d}}$.  Then
 \begin{equation}\label{genkinin1}
 \rho(\cF^{\circ \alpha}\circ\cG^{\circ\beta})\le (\rho(\cF))^\alpha (\rho(\cG))^\beta,\qquad  \alpha,\beta> 0, \alpha+\beta=1.
 \end{equation}
 Assume that $\cF^{\circ \alpha}\circ\cG^{\circ\beta}$ is weakly irreducible.  Let  $\uu=(u_1,\ldots,u_n)\trans$ and $\bv=(v_1,\ldots,v_n)\trans$ be the positive eigenvectors of $\cF$ and $\cG$: $\cF(\uu)=\rho(T)\uu^{\circ(d-1) }, \cG(\bv)=\rho(\cG)\bv^{\circ(d-1)}$.  Then  equality in the above inequality holds if and only if the following conditions are satisfied. There exists $\mathbf{a}=(a_1,\ldots,a_n)\trans >0$ such that
 \begin{equation}\label{eqcondlogconsr}
 f_{i_1,\ldots,i_d}u_{i_2}\cdots u_{i_d}=a_{i_1}g_{i_1,\ldots,i_d}v_{i_2}\cdots v_{i_d} \quad \textrm{for all } i_1,\ldots,i_d\in[n].
 \end{equation}
 \end{lemma}
 \begin{proof}
 Assume that  $\cF$ and $\cG$ are weakly irreducible.   Let $\x=\uu^{\circ \alpha}\circ\bv^{\circ\beta}$.
 H\"older's inequality for $p=\alpha^{-1}, q=\beta^{-1}$ yields
 \begin{eqnarray*}
 \sum_{i_2,\ldots,i_d\in[n]} f_{i_1,\ldots,i_d}^{\alpha}g_{i_1,\ldots,i_d}^{\beta}u_{i_1}^{\alpha}v_{i_1}^{\beta}\ldots u_{i_d}^{\alpha}v_{i_d}^{\beta}=
 \sum_{i_2,\ldots,i_d\in[n]} (f_{i_1,\ldots,i_d}u_{i_1}\ldots u_{i_d})^{\alpha}(g_{i_1,\ldots,i_d}v_{i_1}\ldots v_{i_d})^{\beta}\\
 \le (\cF(\uu)_{i_1})^{\alpha}(\cG(\bv)_{i_1})^{\beta}=(\rho(\cF)^{\alpha}u_{i_1}^{(d-1)\alpha})(\rho(\cG)^{\beta}v_{i_1}^{(d-1)\beta})=
 (\rho(\cF)^{\alpha}\rho(\cG)^{\beta})x_{i_1}^{d-1}.
 \end{eqnarray*}
 So
 \[\rho(\cF^{\alpha}\circ\cG^{\beta})(\x)\le (\rho(\cF)^{\alpha}\rho(\cG)^{\beta})\x^{\circ(d-1)}.\]
 Use \eqref{infmaxcharrho} to deduce \eqref{genkinin1}.

 We now discuss the equality in \eqref{genkinin1}.  Suppose that $\cR:=\cF^{\circ \alpha}\circ\cG^{\circ\beta}$ is weakly irreducible.
 Then $\cF$ and $\cG$ are weakly irreducible.  Assume that equality holds in \eqref{genkinin1}.
 In view of the second part of the characterization \eqref{minmaxrhoirchar} it follows that $\x=\uu^{\circ\alpha}\circ\bv^{\circ\beta}$
 is the eigenvector of $\cR$.  
 The equality case of H\"older inequality yields \eqref{eqcondlogconsr}.
Conversely, if  \eqref{eqcondlogconsr} holds then
$\x$ is a positive eigenvector of $\cR$ corresponding to $\rho(\cR)=\rho(\cF)^{\alpha}\rho(\cG)^{\beta}$.

 To deduce the inequality \eqref{genkinin1} for any nonnegative $\cF,\cG$ we use the continuity argument.  Let $\epsilon>0$ and $0<\cJ_{n,d}\in \R_{ps,+}^{n^{\times d}}$.
 Then \eqref{genkinin1} hods for $\cF(\epsilon):=\cF+\epsilon \cJ_{n,d}, \cG(\epsilon):=\cG+\epsilon\cJ_{n,d}$.  Now let $\epsilon\searrow 0$ to deduce \eqref{genkinin1}.
 \end{proof}

 Let $D\subset \R^m$ be a convex set.  A function $f:D\to \R_+$ is called \emph{logconvex} if the function $\log f:D\to [-\infty,\infty)$
 is continuous and convex.  (Note that if $f$ has value $0$ at some point of $D$ then $f$ is identically zero on $D$.)
 A vector function $\cT:D\to \R^{\m}$ is called logconvex if each entry $t_{i_1,\ldots,i_d}:D\to \R_+$ is logconvex.

 Lemma \ref{genkinin} should be compared with Theorem~4.1 of 
Zhang, Qi, Luo and Xu~\cite{ZQL10}, which states a similar property
under the assumption that $\cF\circ\cG$ is irreducible.
This lemma implies the following generalization of Kingman's theorem for the spectral radius of matrices with logconvex entries.
 \begin{corollary}\label{logconvsr}  Let $D\subset \R^m$ be a convex set.  Assume that $\cT:D\to \R_{ps,+}^{n^{\times d}}$ is logconvex.
 Then $\rho(\cT):D\to\R_+$ is logconvex.
 \end{corollary}

 \section{Generalization of Friedland-Karlin inequality}\label{sec:FKineq}
  For a tensor $\cT=[t_{i_1,\ldots,i_d}]\in \F^{n^{\times d}}$ and a vector $\y=(y_1,\ldots,y_n)\trans \in\F^{n}$ we define $\diag(\y)\circ \cT$ to be the tensor  $[y_{i_1}t_{i_1,\ldots,i_d}]\in \F^{n^{\times d}}$.   Note that if $\cT\in \F^{n^{\times d}}_{ps}$ then  $\diag(\y)\circ \cT\in \F^{n^{\times d}}_{ps}$.
  
In this section we extend the results in \cite[\S6.6]{Fribook} to nonnegative
tensors.  In particular, the following inequality is a generalization of the Friedland-Karlin \cite{FK75} inequality to tensors.
 \begin{theorem}\label{FKtens} Assume that $\cT\in \R_{ps,+}^{n^{\times d}}$ is a weakly irreducible tensor.
 Let $A,\uu,\w$ be given by \eqref{defATu} and \eqref{defpropA}.  Assume that
 $\y=(y_1,\ldots,y_n)\trans>0$.  Then
 \begin{equation}\label{frikar}
 \rho(\diag(\y)\circ \cT)\ge \rho(\cT)\prod_{i=1}^n y_i^{u_iw_i}.
 \end{equation}
Assume furthermore that $\cT$ is a symmetric tensor and {{$\sum_{i=1}^n u_i^d=1$}}. Then
 \begin{equation}\label{frikarsym}
 \rho(\diag(\y)\circ \cT)\ge \rho(\cT)\prod_{i=1}^n y_i^{u_i^d}.
 \end{equation}
  \end{theorem}
\begin{proof}
 For $\x\in\R^n$ let $e^{\x}=(e^{x_1},\ldots,e^{x_n})\trans$.
 Then $\rho(\x)=\rho(\diag(e^{\x})\circ \cT)$ is a log-convex function in $\x$ see \S\ref{sec:logconv}.
 Hence $\lambda(\x):=(\rho(\diag(e^{\x})\circ \cT))^{\frac{1}{d-1}}$ is also log-convex.
 The log-convexity property yields
 $$\log\lambda(\x)\ge \log\lambda(\0)+\x\trans (\differential\log\lambda)(\0)=\log\lambda(\0)+\frac{1}{\lambda(\0)}\x\trans(\differential\lambda)(\0).$$
 Note that $\lambda(\0)=\rho(\cT)^{\frac{1}{d-1}}$.  Now use Theorem \ref{pertspecrad}, together with $e^{x_i}=1+x_i + O(x_i^2)$,
 to show that
 $$\x\trans(\differential\lambda)(\0)=\frac{\lambda_0}{d-1}\sum_{i=1}^n x_iu_iw_i.$$
  So
 $\lambda(x)\ge \lambda_0 e^{\frac{1}{d-1}\sum_{i=1}^n x_iu_iw_i}$.  Raise this to the power $d-1$ and let $y_i=e^{x_i}, i=1,\ldots,n$, 
to deduce
 \eqref{frikar}.

 Assume furthermore that $\cT$ is symmetric.  Then $\w=\uu^{\circ (d-1)}$, where we have the normalization $\sum_{i=1}^n u_i^d=1$.
 Hence \eqref{frikar} is equivalent to \eqref{frikarsym}.
\end{proof}

 Theorem 4.1 in \cite{FK75} claims that if $T$ is an nonnegative irreducible symmetric matrix which is also a positive semi-definite then inequality
 \eqref{frikarsym}  (with $d=2$) can be improved to 
 \[\rho(\diag(\y)\circ T))\ge \rho(T)\sum_{i=1}^d y_i u_i^2.\]
 We now give a generalization of this result.
{{Recall that  that a symmetric tensor $\cT$ is induced by a homogeneous polynomial $f(\x)$ of degree $d$ \ref{subsec:homeigv}}}:
 {{\[f(\x)=\x\trans \cT(\x), \quad \cT(\x) =\cT\times(\otimes^{d-1}\x)=\frac{1}{d}\nabla f(\x)=(F_1(\x),\ldots,F_n(\x))\trans.\]}}
 Observe next that $f(\x)=\x\trans A\x$ corresponds to a positive definite matrix if and only if $f$ is strictly convex.
 Clearly, in this case 
 \begin{equation}\label{poscond}
 f(\x)>0 \textrm{ for all } \x\in\R^n\setminus\{\0\}.
 \end{equation}
Note that if $f(\x)$ is a homogeneous polynomial of degree $d$ then the above condition can hold only if $d$ is an even integer.
\begin{theorem}\label{genKFsc}  Let $d$ be a positive even integer.  Let $f(\x)$ be a homogeneous polynomial of degree $d$  on $\R^n$ and denote by $\cT\in\R_{ps}^{n^{\times d}}$ 
the symmetric tensor induced by $f$.  
Suppose that the following conditions hold: 
\begin{enumerate}
\item The tensor $\cT$  is nonnegative and weakly irreducible.  Suppose furthermore that 
$\cT(\uu)=\rho(\cT)\uu^{\circ (d-1)}$ and $\sum_{i=1}^n u_i^d=1$.
\item Condition \eqref{poscond}.
\item $f(\x)$ is convex on $\R^n$. 
\end{enumerate}
Suppose in addition that $\y\in\R_{++}^n$ (all the coordinates of $\y$ are positive).  Then
\begin{equation}\label{maxcharrhoyT}
  \rho(\diag(\y)\circ \cT)^{\frac{1}{d-1}}=\max_{\x\ne \0} \frac{\sum_{i=1} y_i^{\frac{1}{d-1}}F_i(\x)^{\frac{d}{d-1}}}{f(\x)}.
\end{equation}
In particular
 \begin{equation}\label{frikarsym2}
 \rho(\diag(\y)\circ \cT)\ge \rho(\cT)\left(\sum_{i=1}^n y_i^{\frac{1}{d-1}}{u_i^d}\right)^{d-1}.
 \end{equation}
 If the Hessian of $f(\x)$ is positive definite at each $\x\ne \0$ then equality holds iff $\y=c\1$.
\end{theorem} 
\begin{proof}
  {{Let $F_{ij}=\frac{1}{d-1}\frac{\partial F_i}{\partial x_j}=\frac{1}{(d-1)d}\frac{\partial ^2\;\;\;\;\;\;}{\partial x_j\partial x_i}f$.}}
We first assume that the Hessian $H(\x)$ of $f(x)$ is positive definite for each $\x\ne \0$.  Assume that $\y=(y_1,\ldots,y_n)\trans>\0$ is fixed.
We now consider the critical points of the ratio
$\frac{\sum_{i=1} y_i^{\frac{1}{d-1}}F_i(\x)^{\frac{d}{d-1}}}{F(\x)}$ for $\x\ne \0$.

Clearly, every critical point of the above ratio satisfies
 \begin{equation}\label{critpointcond}
 \sum_{i=1}^n y_i^{\frac{1}{d-1}} F_i(\x)^{\frac{1}{d-1}}F_{ij}(\x)=\lambda F_j(\x), \quad j\in [n], \x\ne \0,\lambda\ne 0.
 \end{equation}
Next we consider the following the eigenvalue problem for $\diag(\y)\circ\cT$:
 \begin{equation}\label{eigenvector}
 y_iF_i(\x)=\mu \x^{\circ(d-1)}, \x\in\R^n\setminus\{\0\}.
 \end{equation}  
 Observe first that $\mu>0$. Indeed, Euler's formula yields that 
 \[0<dF(\x)=\sum_{i=1}^d x_i F_i(\x)=\mu \sum_{i=1}^d y_i^{-1}x_i^d.\] 
 We claim that each eigenvector satisfying \eqref{eigenvector} satisfies \eqref{critpointcond}.
 Indeed, 
 \[F_i(\x)^{\frac{1}{d-1}}= y_i^{-\frac{1}{d-1}}\mu^{\frac{1}{d-1}}x_i, \quad i\in[n].\]
 Observe next that $F_j$ is a homogeneous function of degree $d-1$.  Furthermore $F_{ij}=F_{ji}$.
 Use Euler's formula to deduce
 \[ \sum_{i=1}^n y_i^{\frac{1}{d-1}} F_i(\x)^{\frac{1}{d-1}}F_{ij}(\x)=\mu^{\frac{1}{d-1}}\sum_{i=1}^n x_iF_{ji}(\x)=\mu^{\frac{1}{d-1}}F_j(\x).\]
 Hence \eqref{critpointcond} holds with $\lambda=\mu^{\frac{1}{d-1}}$.
 
 Assume now that \eqref{critpointcond} holds.  We claim that \eqref{eigenvector} holds with $\mu=\lambda^{d-1}$.
 Indeed, Euler's identities yield that
 \[\sum_{i=1}^n \lambda x_i F_{ij}(\x)=\lambda F_j(\x), \quad j\in[n].\]
 Since $H(\x)=[F_{ij}(\x)]$ is invertible it follows that  $y_i^{\frac{1}{d-1}} F_i(\x)^{\frac{1}{d-1}}=\lambda x_i$ for $i\in[d]$.
 It is left to show that the maximum $\lambda$ is $\rho(\diag(\y)\circ)^{\frac{1}{d-1}}$.  Indeed, consider the system \eqref {eigenvector}.
 Clearly
 \[\mu |x_i|^{d-1}\le y_i F_i(|\x|) \textrm{ for } i\in [n].\]
The Collatz-Wielandt maximin characterization in Equation~\eqref{minmaxrhowirchar} yields that 
\[ 
\mu\le \rho(\diag(\y)\circ\cT) \enspace .
\]
 As $\nabla F(\x)=d\cT(\x)$ it follows that the maximum critical value of $\mu$ is $d\rho(\diag(\y)\circ \cT)$.  This shows \eqref{maxcharrhoyT}.
 
 To show \eqref{frikarsym2}  choose $\x=\uu$ in the maximum characterization  \eqref{maxcharrhoyT}.  Since $\diag(\y)\circ\cT$ is weakly irreducible it follows that
 equality in \eqref{frikarsym2} is achieved if and only if $\uu$ is the Perron-Frobenius eigenvalue of $\diag(\y)\circ\cT$, i.e., $\y=c\1$.
  
 We now show  \eqref{maxcharrhoyT} and \eqref{frikarsym2} assuming that $F(\x)$ is convex but no longer necessarily strictly convex. 
So $H(\x)$ is a positive  semi-definite symmetric
 matrix.  Consider 
 \[G(\x)=\sum_{i=1}^{2n-1} (\sum_{j=1}^n b_{ij}x_j)^d.\] 
 Let $B=[b_{ij}]\in\R^{(2n-1)\times n}_{+}$.  Assume that any $n$ rows of $B$ are linearly independent.  It is straightforward to show that $G(\x)$ 
 satisfies all the assumptions of the theorem.  Moreover $H(G)(x)$ is positive definite for $\x\ne \0$.  Let $\cS$ be the induced symmetric tensor by $G(\x)$.
 Then for $\epsilon >0$, $\cT(\epsilon)=\cT+\epsilon \cS$ satisfies the assumptions of the theorem.  Furthermore $H(\cT(\epsilon))(\x)$ is positive definite
 for $\x\ne 0$.  Hence the characterizations   \eqref{maxcharrhoyT} and \eqref{frikarsym2} hold.  Letting $\epsilon \to 0$ we deduce the theorem for $\cT$.
  \end{proof} 
  
   \begin{rem} The arguments of the proof of the Theorem \ref{genKFsc} apply if we replace the condition \emph{3} of Theorem  \ref{genKFsc}
   by the condition that the Hessian of $F(\x)$ is invertible for each $\x\ne \0$.  For $d=2$, the condition \emph{2}, i.e., \eqref{poscond}, yields that $F$ is strictly convex.
   We do not know if the condition  \eqref{poscond} and the condition that  the Hessian of $F(\x)$ is invertible implies that $F$ is strictly convex for an even $d>2$.
\end{rem}

 We now generalize another inequality in \cite[Theorem 3.1]{FK75}.
 \begin{proposition}\label{corFKtens}  
Assume that $\cT\in \R_{ps,+}^{n^{\times d}}$ is a weakly irreducible tensor.
Then
 \begin{equation}\label{minrhoTchar}
 \log \rho(\cT)=\min_{\x>0} \sum_{i=1}^n u_iw_i \log \frac{(\cT\x)_i}{x_i^{d-1}}.
 \end{equation}
The equality holds if and only if $\x$ is the eigenvector $c\uu, c>0$ of $\cT$.
 \end{proposition} 
 \begin{proof}  Let $\x>\0$.  Define $y_i=\frac{x_i^{d-1}}{\cT(\x)_i}$.
 Then $(\diag(\y)\cT)(\x)=\x^{\circ(d-1)}$.  So $\rho(\diag(\y)\cT)=1$.  Apply \eqref{frikar} to deduce the inequality
 $ \sum_{i=1}^n u_iw_i \log \frac{(\cT\x)_i}{x_i^{d-1}}\ge\log \rho(\cT)$.  Equality holds if and only if $\y=c\1$, i.e.\
 $\x=c\uu$.  \end{proof}

 We now generalize the finite dimensional version of the Donsker-Varadhan inequality \cite{DV75} as in \cite{Fri81}.  Denote by $\Pi_n\subset \R_+^n$ the set of probability vectors $\p=(p_1,\ldots,p_n)\trans$.
 \begin{theorem}\label{DVtens}
Assume that $\cT\in \R_{ps,+}^{n^{\times d}}$.
Then
 \begin{equation}\label{DVten}
 \log\rho(\cT)=\max_{(p_1,\ldots,p_n)\trans\in \Pi_n}\inf_{\x>\0}\sum_{i=1}^n p_i  \log\frac{(\cT\x)_i}{x_i^{d-1}}.
 \end{equation}
 \end{theorem}
\begin{proof}
Recall that Sion's theorem~\cite{sion} shows that
\[\max_{p\in P}\inf_{y\in Y}L(p,y)=\inf_{y\in Y}\max_{p\in P}L(p,y)
\]
if $P$ is a convex compact subset of $\R^N$, $Y$ is a convex subset of $\R^M$,
for all $\y\in Y$, $\p\mapsto L(\p,\y)$ is concave and upper semi-continuous, and for all $\p\in P$, $\y\mapsto L(\p,\y)$ is convex and lower semicontinuous.
Let us apply this result to $P:=\Pi_n$, $B=\R^n$, 
\[ L(p,y) = \sum_{i=1}^n p_i \log((\cT\exp(y))_i/\exp((d-1)y_i)
\enspace . 
\]
Here, the map $p\mapsto L(p,y)$ is linear, whereas
the convexity of the map $y\mapsto L(p,y)$ follows from the
fact that the set of log-convex functions is a convex cone~\cite{Kin61}.
By the Collatz-Wielandt formula~\eqref{infmaxcharrho}, 
\begin{eqnarray}
\log \rho(\cT)&=
\inf_{y\in\R^n}\max_{i\in[n]}\log((\cT\exp(y))_i/\exp((d-1)y_i))\nonumber
\\
&= \inf_{y\in\R^n}\max_{p\in\Pi_n}\sum_{i\in [n]}p_i\log((\cT\exp(y))_i/\exp((d-1)y_i)) \enspace.
\nonumber
\end{eqnarray}
By Sion's theorem, we obtain~\eqref{DVten}.
\end{proof}

 The following theorem is a generalization of \cite[Theorem  3.3]{Fri81}.  The proof is identical to the proof  in \cite{Fri81}, in which the theorem is deduced from the special case of Theorem~\ref{DVtens} concerning nonnegative matrices, 
so we omit it. 
 \begin{theorem}\label{thm3.3F}  
Assume that $\cT\in \R_{ps,+}^{n^{\times d}}$.
Assume that $\Psi:\R\to \R$ is a convex function.  
Suppose furthermore that $\Psi'(\log \rho(\cT))\ge 0$.  Then
 \[\Psi(\log\rho(\cT))=\max_{(p_1,\ldots,p_n)\trans\in \Pi_n}\inf_{\x>\0}\sum_{i=1}^n p_i  \Psi(\log\frac{(\cT\x)_i}{x_i^{d-1}}).\]
 In particular
 \begin{equation}\label{DVten1}
 \rho(\cT)=\max_{(p_1,\ldots,p_n)\trans\in \Pi_n}\inf_{\x>\0}\sum_{i=1}^n p_i  \frac{(\cT\x)_i}{x_i^{d-1}}.
 \end{equation}
 \end{theorem}
 The last inequality is a generalization of the finite dimensional version of the Donsker-Varadhan inequality.
 The following result is a generalization of J.E. Cohen's result for matrices \cite{Coh79}.  See \cite[Theorem 3.1]{ZQL10}. 
 \begin{corollary}\label{convrhodiagel}  The spectral radius of a tensor
 $\cT=[t_{i_1,\ldots,i_d}]\in\R_{ps,+}^{n^{\times d}}$ is a convex function
 in the diagonal entries $(t_{1,\ldots,1},\ldots,t_{n,\ldots,n})\trans\in\R_+^n$. 
 \end{corollary}
 \begin{proof} We showed that the spectral radius depends continuously
on the entries of the tensor. Hence, arguing by density, we may assume that $\cT$ is weakly irreducible.
Let $\cT'=[t_{i_1,\ldots,i_d'}]\in\R_{ps,+}^{n^{\times d}}$ where the diagonal entries 
 of $\cT'$ are zero, while nondiagonal entries are equal to the corresponding entries of $\cT$.  Then 
 \[\inf_{\x>\0 } \sum_{i=1}^n p_i\frac{\cT(\x)_i}{x_i^{d-1}}=\sum_{i=1}^n p_i t_{i,\ldots,i} +\inf_{\x>\0 } \sum_{i=1}^n p_i\frac{\cT'(\x)_i}{x_i^{d-1}}.\]
 Hence the supremum over $\Pi_n$ is a convex function in the diagonal entries.  \end{proof}
 
 We close this section with the following generalization of \cite[Theorem 3.2]{FK75}.  Suppose that $A\in\R_+^{n\times n}$ is irreducible and all diagonal entries
 are positive.  Let $\uu,\w$ two positive vectors in $\R^n$.  Then there exists a matrix $B$ diagonally equivalent to $A$, i.e.\ $B=\diag(e^{\y})A\diag(e^{\z})$
 such that $B\uu=\uu, B\trans \w=\w$.
 
 We say that tensors $\cT=[t_{i_1,\ldots,i_d}],\cT'=[t_{i_1,\ldots,i_d}'],\in \R_{ps,+}^{n^{\times d}}$ are diagonally equivalent if
 $t'_{i_1,\ldots,i_d}=t_{i_1,\ldots,i_d}e^{b_{i_1}+\sum_{j=2}^d c_{i_j}}$ for all $i_1,\ldots,i_d\in [n]$ for some $\mathbf{b}=(b_1,\ldots,b_n)\trans, \mathbf{c}=(c_1,\ldots,c_n)\trans\in\R^n$.
 \begin{theorem}\label{diagteneqthm}  Let $\cT'\in \R_{ps,+}^{n^{\times d}}$ be an irreducible tensor with positive diagonal entries.
 Let $\uu,\w$ be two given positive vectors in $\R^n$ satisfying $\sum_{i=1}^n u_iw_i=1$. 
 Then there exists a diagonal equivalent tensor $\cT$ to $\cT'$ which satisfies the following conditions.  First, $\cT(\uu)=\uu^{\circ (d-1)}$.
 Second \eqref{defpropA} holds with $\rho(\cT)=1$.
 \end{theorem}
 \begin{proof}  Consider the convex function $f(\x,\cT'):=\sum_{i=1}^n u_iw_i (\log\cT'(e^{\x})_i -(d-1)x_i)$ on the hyperplane $H=\{\x\in\R^n,\sum_{i=1} x_i=0\}$.
 Observe that since the diagonal entries of $\cT$ are positive we obtain that each expression $\cT'(e^{\x})_i e^{-(d-1)x_i}\ge t_{i,\ldots,i}$.  That is $ u_iw_i (\log\cT'(e^{\x})_i -(d-1)x_i)\ge u_iw_i \log t_{i,\ldots,i}$ for each $i\in [n]$.
 We claim, as in \cite{FK75}, that $\lim _{k\to\infty}f(\x_k,\cT')=\infty$ for any sequence $\x_k=(x_{1,k},\ldots,x_{n,k})\trans\in H$ such that $\lim_{k\to\infty} \|\x_{k}\|=\infty$. 
 Indeed by taking a subsequence and renaming the coordinates $x_1,\ldots,x_n$ we can assume the following conditions.  First $x_{1,k}\le \ldots\le x_{n,k}$ for each positive integer $k$.  Furthermore, there there exists $l\in [n-1]$ such that $\lim_{k\to\infty}x_{i,k}=-\infty$ for $i\in[l]$, and $x_{l+1,k}\ge a (\in\R)$  for each positive integer $k$. Since $\cT'$ is irreducible there $i\in[l-1]$ and $j_1,\ldots,j_{d-1}\in[n]\setminus [l-1]$ such that $t_{i,j_1,\ldots,,j_{d-1}}>0$.  Hence
$\cT'(e^{\x_k})_i e^{-(d-1)x_{i,k}}\ge t_{i,j_1,\ldots,j_{d-1}}e^{(d-1)(a-x_{i,k})}$.  Thus
$\lim_{k\to \infty}\cT'(e^{\x_k})_i e^{-(d-1)x_{i,k}}=\infty$, which implies that  $\lim _{k\to\infty}f(\x_k,\cT')=\infty$.  Hence $f(\x,\cT')$ achieves its minimum at some critical point $\y\in H$. Let $\1=(1,\ldots,1)\trans$.  Observe that $f(\x,\cT')=f(\x+t\1)$ for any $t\in\R$.
Thus the minimum of $f(\x,\cT')$ on $\R^n$ is achieved at each point of the form $\y+t\1$.
We now study the effects of rescaling of $\cT'$.  First, consider the rescaling $\tilde\cT=[\tilde t_{i_1,\ldots,i_d}]$, where $\tilde t_{i_1,\ldots,i_d}=e^{a_{i_1}} t_{i_1,\ldots,i_d}$ for some $\mathbf{a}=(a_1,\ldots,a_n)\trans\in\R^n$.  Then the minimum of $f(\x,\tilde\cT)$ on $\R^n$ is achieved at $\y+t\1$.  Second,  consider the rescaling $\hat\cT=[\hat t_{i_1,\ldots,i_d}]$, where $\hat t_{i_1,\ldots,i_d}=e^{\sum_{j=2}^d b_{i_j}} t_{i_1,\ldots,i_d}$ for some $\mathbf{b}=(b_1,\ldots,b_n)\trans\in\R^n$.  Then the minimum of $f(\x,\hat\cT)$ is achieved at the points $\y-\mathbf{b}+t\1$.  Now choose $\mathbf{b}=\y-\log\uu$.  Then the minimum  of $f(\x,\hat\cT)$ is achieved at the point $\log\uu$.  Finally, rescale $\hat\cT$ to obtain $\cT=[t_{i_1,\ldots,i_d}]$, where $t_{i_1,\ldots,i_d}=e^{a_{i_1}} t_{i_1,\ldots,i_d}$ for a unique $\mathbf{a}=(a_1,\ldots,a_n)\trans\in\R^n$ such that $\cT(\uu)=\uu^{\circ (d-1)}$.   In particular $\rho(\cT)=1$ and \eqref{pratTuid} holds.  Therefore the first equality of \eqref{defpropA} holds.  As $\log\uu$ is a minimal point of $f(\x,\cT)$ we deduce by  straightforward calculations that the second equality of \eqref{defpropA} holds. \end{proof}
 
 See the paper \cite{TFL11} which gives some new applications to \cite[Theorem 3.2]{FK75}.

\section{Entropic characterization of the spectral radius}\label{sec:entropsr}
\subsection{Entropic characterization of the spectral radius of a nonnegative matrix}
A nonnegative matrix $\mu=[\mu_{ij}]\in \R_+^{n \times n}$ is called an {\em occupation 
measure} 
if the following conditions are satisfied:
\begin{equation}\label{defmatocme}
\sum_{i,j=1}^n \mu_{ij}=1, \quad \sum_{j=1}^n \mu_{ij}=\sum_{j=1}^n\mu_{ji}, \textrm{ for all }i\in[n], \;\mu\in \R_+^{n \times n}.
\end{equation}
There is a natural interpretation of an occupation measure in terms of weights on the directed graph $\kdigraph_n$ on the set of vertices $[n]$.
Assume that the weight of each diedge $(i,j)$, the edge from $i$ to $j$ is the $\mu_{ij}$.  The first condition of \eqref{defmatocme} means that $\mu$ is a probability 
measure on the $n^2$ diedges of $\kdigraph_n$.  The second condition of \eqref{defmatocme} can be easily explained in terms of flow (circulation), whose value on the diedge
$(i,j)$ is $\mu_{ij}$.
Namely, for each vertex $i$ the sum of the flow out of the vertex $i$ is equal to the sum of the flow into the vertex $i$.

A sequence $\gamma$ of diedges is called a {\em dicycle  of length} $k$ in $\kdigraph_n$, if
 there exists $k$ distinct vertices $i_1,\ldots,i_k\in [n]$ such that the $k$ diedges of $\gamma$ are  $(i_1,i_2)$, \ldots, $(i_{k-1},i_k),(i_k,i_1)$.
 A dicycle of length one is the edge $(i_1,i_1)$.  It will be convenient to denote the edges of the dicycle $\gamma$ as $(i_j,i_{j+1}), j\in [k]$,
 where $i_{k+1}=i_1$.  Denote by $\Sigma_n$ the collections of all dicycles in $\kdigraph_n$.
 
 To each cycle $\gamma$ we associate the following occupation measure $\mu(\gamma)$.  Assume that the length of the cycle is $k$.
 then the weight of each edge in the cycle $\gamma$ is $\frac{1}{k}$.  Other edges have weight zero. In other words, $\mu(\gamma)$ represents the frequency
of visit of the edges, in an infinite walk obtained by repeating the cycle $\gamma$. %

Denote by $\Omega(n)\subset\R_+^{n \times n}$ the compact convex set of occupation measure.  
For a subset $S\subseteq [n]\times [n]$ denote by $\Omega(n, S)\subseteq \Omega(n)$ the subset of all occupation measures whose support is contained in $S$.
The following is well known, we provide a proof for completeness. 
\begin{lemma}\label{extptdOmega}  The extreme points of $\Omega(n)$ are the occupation measures $\mu(\gamma)$, where $\gamma\in \Sigma_n$.
Let $S\subset [n]\times [n]$.  Then $\Omega(n, S)\ne \emptyset$ if and only if $S$ contains a dicycle.  Suppose that $S$ contains a dicycle.  Then $\Omega(n, S)$
is a nonempty compact convex set, whose extreme point are $\mu(\gamma)$, where $\gamma$ are all dicycles in $S$.
\end{lemma}
\begin{proof}  We first prove that if $\mu\in \Omega(n)$ then the support of $\mu$ contains a dicycle $\gamma$.  Assume to the contrary that it is not the case.
Since $\mu$ is a probability measure on $\kdigraph_n$ it follows that there exists $\mu_{i_1i_2}>0$. As the support of $\mu$ does not contain a dicycle we have that $i_1\ne i_2$.  The first condition of \eqref{defmatocme} for $i=i_2$ implies that there exists $i_3$ such that $\mu_{i_2i_3}>0$.  Since the support of does not contain a cycle we get that $i_3\notin \{i_1,i_2\}$. Continuing in this manner we deduce \ that in the step $k$ we have $k+1$ distinct in indices $i_1,\ldots,i_{k+1}$ such that $\mu_{i_pi_{p+1}}>0$ for $p\in[k]$.  For $k=n$ we obtain the contradiction.

We now show that the convex set spanned by the set $E(n)=\{\mu(\gamma), \gamma\in\Sigma_n\}$ is $\Omega(n)$.
For $p\in[n^2]$ denote by $\Omega_p(n)$ the subset of all occupation measures with at most $p\ge 1$ nonzero entries.
We show by induction that the convex set spanned by $E(n)$ contains $\Omega_p(n)$.  For $p=1$ the set $\Omega_p(n)$ consists of all $\mu(\gamma)$, where
$\gamma$ is a dicycle of length one.  Suppose that the claim holds for $p\le q$.  Assume that $p=q+1$.  Let $\mu\in \Omega(n)$ has exactly $q+1$ nonzero entry.
Assume a dicycle $\gamma$ in the support of $\mu$.  If $\mu=\mu(\gamma)$ we are done.  Otherwise,
let $a>0$ be the maximal $b>0$ such that $\mu -b\mu(\gamma)\ge 0$.   Then $\mu_1:=\frac{1}{1-a}(\mu-\mu(\gamma))\in \Omega_{q}(n)$.  So $\mu_1$ is a convex combination
of measures in $E(n)$.  As $\mu=(1-a)\mu_1+a\mu(\gamma)$ it follows that $\mu$ is a convex combination of some $\mu(\gamma)$.
Hence the set of the extreme points of $\Omega(n)$ is contained in $E(n)$.  Clearly, $\mu(\gamma)$ is not a convex combination of the measures in $E(n)\setminus\{\mu(\gamma)\}$.  Hence $E(n)$ is the set of the extreme points of $\Omega(n)$.

 The other claims of the lemma follow straightforwardly from the above arguments.\end{proof} 

Occupation measures are closely related to stochastic matrices:
\begin{lemma}\label{ocmstocmat}  Denote by $\operatorname{Stoc}(n)\subset \R_+^{n\times n}$ the convex set of (row) stochastic matrices.  Then there exists  a map  $\Psi_n: \Omega(n)\to \operatorname{Stoc}(n)$ and a multivalued map
$\Phi_n: \operatorname{Stoc}(n)\to\Omega(n)$
with the following properties.
\begin{enumerate}
\item For each $A\in \operatorname{Stoc}(n)$ the set $\Phi_n(A)$ is a closed nonempty 
convex set of occupation measures. 
\item $\Phi_n(A)$ consists of a unique occupation measure if and only if $1$ is a simple root of $\det(zI-A)$.
\item If $A\in \operatorname{Stoc}(n)$ is irreducible then $\Phi_n(A)$ consists of a unique occupation measure $\mu(A)$ which is irreducible.  Furthermore, $\Psi_n(\mu(A))=A$. 
\item If $\mu\in\Omega(n)$ is irreducible then
$\Psi_n(\mu)$ is irreducible  and  $\Phi_n(\Psi_n(\mu))=\{\mu\}$.
\item For each $\mu\in\Omega(n)$ the convex set $\Phi_n(\Psi_n(\mu))$ contains $\mu$.

 \end{enumerate}
 \end{lemma} 
 \begin{proof}  Assume that $A\in \operatorname{Stoc}(n)$.  Let $\z=(z_1,\ldots,z_n)\trans$ be the stationary distribution corresponding to $A$.  So $\z$ is a probability vector satisfying $A\trans\z=\z$.  A straightforward computation shows that $\diag(\z)A\in\Omega(n)$.
We define $\Phi_n(A)$ to be the set of all occupation measures of this form. 
   Hence $\Phi_n(A)$ is a closed convex set.  This proves part \emph{1}.
Furthermore $\Phi_n(A)$ consists of one occupation measure $\mu(A)$ if and only if $\z$ is unique, i.e., $1$ is a geometrically  simple eigenvalue.  
It is a classical property of stochastic matrices that the geometric and algebraic multiplicity of the eigenvalue $1$ coincide, see Theorem 6.5.3 in \cite{Fribook}; hence, $1$, is an algebraically simple eigenvalue. 
This proves part \emph{2}.
Clearly, if $A$ is irreducible then $\z>\0$ is unique and $\mu(A)$ is irreducible.

We now define $\Psi_n(\mu)$.  Suppose first that $\mu$ does not have a zero row.  Let $r_i=\sum_{j=1}^n \mu_{ij}$.  Then $\Psi_n(\mu):=\diag(r_1^{-1},\ldots,r_n^{-1})\mu$.  Note that $\z=(r_1,\ldots,r_n)\trans$ is a probability vector satisfying
$\Psi_n(\mu)\trans \z=\z$.  Hence $\Phi_n(\Psi_n(\mu))$ contains $\mu$.  Clearly, if $\mu$ is irreducible then $\Psi_n(\mu)$ is irreducible.  Parts \emph{3} and \emph{4} follow straightforwardly. 

Assume now that $\mu$ has zero rows.  Let $S(\mu)\subset [n]$ be the subset of all zero rows of $\mu$.  As $\mu$ is an occupation measure, $S(\mu)$ is also the subset of zero columns of $\mu$.  Let $k$ be the cardinality of $S(\mu)$.  Then $\mu$ is a direct sum of $\mu_1\oplus 0_{k\times k}$, where $\mu_1\in\Omega(n-k)$ with nonzero rows, and
$0_{k\times k}$ is the $k\times k$ zero matrix.  Let $J_k\in\R^{k\times k}$ be the matrix whose all entries are $1$.  Then $\Psi_n(\mu)=\Psi_{n-k}(\mu_1)\oplus \frac{1}{k}J_k$.
Clearly, $\Phi_n(\Psi_n(\mu))$ contains $\mu$ in this case.  This completes the proof of part \emph{5}.\end{proof}

Assume that $A=[a_{ij}]\in\R_+^{n\times n}$.
Denote by $\supp{A}\subset [n]\times [n]$ the set of $(i,j)\in [n]\times [n]$ such that $a_{ij}>0$. It is well known that  $\rho(A)=0$
if and only if $\supp{A}$ does not contain a dicycle.  (This follows from the Frobenius normal form of $A\in\R_{+}^{n\times n}$ \cite[Theorem 6.4.4]{Fribook}.)  Let $\mu\in \Omega(n)$.  Denote by $S(\mu)$ the set of zero rows of $\mu$.  Then $\supp{\mu}\subseteq ([n]\setminus S(\mu))^2$.
The following result characterizes $\log\rho(A)$ for a nonnilpotent $A=[a_{ij}]\in\R_+^{n\times n}$ as the value of an entropy maximization problem.
\begin{theorem}\label{AGchar}  Let $A=[a_{ij}]\in\R_+^{n\times n}$.  Then
\begin{equation}\label{AGchar1}
\log\rho(A)=\max_{\mu=[\mu_{i,j}]\in\Omega(n)} \sum_{i,j\in[n]} \mu_{ij}\log \frac{a_{ij}\sum_{k=1}^n \mu_{ik}}{\mu_{ij}}.
\end{equation}
\end{theorem}
As usual $0\log 0=0$ and $t\log 0=-\infty$ for $t>0$. In particular, observe
that the term in the maximum is equal to $\infty$ if $a_{ij}=0$ and $\mu_{ij}>0$ for some $(i,j)$, so in~\eqref{AGchar1}, the maximum can be restricted
to those occupation measures $\mu\in \Omega(n)$ whose support is included
in the support of $A$.
Formula~\eqref{AGchar1} characterizes the logarithm of the spectral radius
as the value of a discrete ergodic control problem. Results
of this nature have appeared in risk sensitive control,
see Theorem~3 of~\cite{AB17}.
We shall explain the control interpretation in \Cref{subsec-ergodic}. 
We next provide a proof from first principles, not relying on ergodic control,
as this will serve in the extension to nonnegative tensors, in \S\ref{subsecmain}.
We start with the following lemma.
\begin{lemma}\label{auxlemAG}  Let $A=[a_{ij}]\in\R_+^{n\times n}, \mu=[\mu_{ij}]\in\Omega(n)$.  Assume that $A$ and $\mu$ are irreducible and $A$ and $\mu$ have the same supporting set in $[n]\times [n]$.  Then
\begin{equation}\label{auxlemAG1}
\log\rho(A)\ge \sum_{i,j=1}^n \mu_{ij}\log \frac{a_{ij}\sum_{k=1}^n \mu_{ik}}{\mu_{ij}}.
\end{equation}
Equality holds if and only if $\mu$ of the form 
\begin{equation}\label{maxmu}
\mu=[\mu_{ij}], \;\mu_{ij}=\frac{1}{\rho(A)} w_ia_{ij}u_j \textrm{ for } i,j\in[n], 
\end{equation}
where 
\[A\uu=\rho(A)\uu,\; \w\trans A=\rho(A)\w\trans,\;\uu,\w>\0, \w\trans\uu=1.\]
\end{lemma}
\begin{proof} Let $x_i=\sum_{j=1}^n \mu_{ij},j\in[n]$ and $\x=(x_1,\ldots,x_n)\trans$.  Note that
$\nu=[\nu_{ij}]=[\frac{1}{x_i}\mu_{ij}]$ is a row stochastic irreducible matrix, where $\x\trans \nu=\x\trans$ and $\x\trans\1_n=1$.   Consider a log-convex map: $t\to C(t):=[\nu_{ij}e^{b_{ij}t}]$, where $b_{ij}=0$ if $\nu_{ij}=0$.  Corollary \ref{logconvsr} yields that $\log\rho(C(t))$ is a convex function.  As $\mu$ was irreducible, it follows that $C(t)$ is irreducible. Clearly 
\[C(t)=\nu +\sum_{k=1}^{\infty} \frac{t^k}{k!} \nu\circ B^{\circ k}.\]
Hence by the standard variation formula for an algebraically simple eigenvalue $1$ of $\nu$ \cite[\S3.8]{Fribook}:
\[(\log\rho(C(t))'(t=0)=\frac{1}{\rho(\nu)}\rho(C(t))'(t=0) =\x\trans (\nu\circ B)\1_n=\sum_{i,j=1}^n \mu_{ij} b_{ij}.\]
Now choose $b_{ij}=\log \frac{x_i a_{ij}}{\mu_{ij}}$ if $a_{ij}>0$.  Then \eqref{auxlemAG1} follows from the convexity of $\log\rho(C(t))$:
\[\log \rho(C(1))\ge \log\rho(C(0))+(\log\rho(C(t))'(t=0)=\sum_{i,j=1}^n \mu_{ij}\log \frac{a_{ij}\sum_{k=1}^n \mu_{ik}}{\mu_{ij}}.\]
Let $\mu$ be given by \eqref{maxmu}.   Observe that 
\[\sum_{j=1}^n \frac{1}{\rho(A)} w_ia_{ij}u_j=\sum_{j=1}^n \frac{1}{\rho(A)}w_j a_{ji}u_i=w_iu_i \quad \textrm{ for } i\in[n].\]
As $\w\trans \uu=1$ it follows that $\mu\in\Omega(n)$.  Clearly, $\supp A= \supp \mu$.
We claim that equality holds in \eqref{auxlemAG1}.   The above equalities yield 
\[\frac{a_{ij}\sum_{k=1}^n \mu_{ik}}{\mu_{ij}}=\rho(A)\frac{w_iu_i}{w_iu_j}=\rho(A)\frac{u_i}{u_j}.\]
Hence
\begin{eqnarray*}
&&\sum_{i,j=1}^n \mu_{ij}\log\frac{a_{ij}\sum_{k=1}\mu_{ik}}{\mu_{ij}}=\sum_{i,j=1}^n \mu_{ij}\log\rho(A) +\sum_{i,j=1}^n \mu_{ij} (\log u_i - \log u_j)=\\
&&\log\rho(A)+\left(\sum_{i=1}^n \log u_i\sum_{j=1}^n \mu_{ij}-\sum_{j=1}^n \log u_j\sum_{i=1}^n \mu_{ij}\right)=\\
&&\log\rho(A) +\left(\sum_{i=1}^n w_iu_i\log u_i\right) -
\left(\sum_{j=1}^n w_ju_j\log u_j\right)=\log\rho(A).
\end{eqnarray*}
 
{It is left to show that $\log\rho(C(t))$ is strictly convex at on the interval $[0,1]$ unless $\nu_{ij}=\rho(A)^{-1}u_i^{-1}a_{ij}u_j$ for $i,j\in[n]$.  Set $F=\nu$ and $G=C(1)=A$ for $t\in (0,1]$ and use Lemma \ref{genkinin}.  Assume that  equality holds in \eqref{genkinin1}.  Hence \eqref{eqcondlogconsr} holds.  Recall that $\nu \1_n=\1_n$ and $A\uu=\rho(A)\uu$. Hence $\nu_{ij}=s_ia_{ij}u_j$ for some $s_1,\ldots,s_n>0$.  As $\nu\1_n=\1_n$ it follows that $s_i=\rho(A)^{-1}u_i^{-1}$ for $i\in [n]$. A straightforward calculation shows that $\z=(u_1w_1,\ldots,u_nw_n)\trans$ is the left probability eigenvector of $\nu$ corresponding to $1$.  Therefore \eqref{maxmu} holds. }
\end{proof}

\begin{proof}[Proof of Theorem \ref{AGchar}]  Assume first that $A>0$.  The for each $\mu>0$ we have inequality \eqref{auxlemAG1}.  Hence
\begin{eqnarray*}
\log\rho(A)\ge\sup_{\mu=[\mu_{i,j}]\in\Omega(n), \mu>0} \sum_{i,j\in[n]} \mu_{ij}\log \frac{a_{ij}\sum_{k=1}^n \mu_{ik}}{\mu_{ij}}=\\ \sup_{\mu=[\mu_{i,j}]\in\Omega(n)} \sum_{i,j\in[n]} \mu_{ij}\log \frac{a_{ij}\sum_{k=1}^n \mu_{ik}}{\mu_{ij}}.
\end{eqnarray*}
Choose $\mu$ as in \eqref{maxmu} to deduce \eqref{AGchar1}.

Assume now that $A\ge 0$ but not positive. 
First observe that if $a_{ij}=0$ and $\mu_{ij}>0$ then 
$\mu_{ij}\log\frac{a_{ij}\sum_{k=1}^n\mu_{ik}}{\mu_{ij}}=-\infty$.  Hence $\sum_{i,j=1}^n 
\mu_{ij}\log\frac{a_{ij}\sum_{k=1}^n\mu_{ik}}{\mu_{ij}}=-\infty$. 
Suppose that $\rho(A)=0$.  Then $\log\rho(A)=-\infty$.  Since a support of $A$ does not contain a dicycle we deduce that $\sum_{i,j=1}^n 
\mu_{ij}\log\frac{a_{ij}\sum_{k=1}^n\mu_{ik}}{\mu_{ij}}=-\infty$ for each $\mu\in\Omega(n)$.  Therefore \eqref{AGchar1} holds in this case.

Suppose that $A=[a_{ij}]$ is irreducible.  So $\log\rho(A)>-\infty$.  The above arguments imply that it is enough to show
\begin{equation}\label{AGchar2}
\log\rho(A)=\max_{\mu=[\mu_{i,j}]\in\Omega(n,\supp{A})} \sum_{i,j\in[n]} \mu_{ij}\log \frac{a_{ij}\sum_{k=1}^n \mu_{ik}}{\mu_{ij}}.
\end{equation}
For $\mu\in\Omega(n,\supp{A})$ such that $\supp{A}=\supp{\mu}$ we can use Lemma \ref{auxlemAG} to deduce \eqref{AGchar2} as for $A>0$.  

It is left to show for \eqref{AGchar1}  for a nonnilpotent nonirreducible $A$.
Let $J_n\in\R^{ns n}_+$, where each entry of $J_n$ is $1$.  Consider $A(\varepsilon)=A+\varepsilon J_n$, where $\varepsilon>0$.  Then $\rho(A)<\rho(A(\varepsilon)$.  As $\log a_{ij}<\log(a_{ij}+\varepsilon)$, and the theorem holds for $A(\varepsilon)$, it follows that
\[\log\rho(A(\varepsilon)) >\sup_{\mu=[\mu_{i,j}]\in\Omega(n)} \sum_{i,j\in[n]} \mu_{ij}\log \frac{a_{ij}\sum_{k=1}^n \mu_{ik}}{\mu_{ij}}.\]
Letting $\varepsilon\searrow 0$ we deduce the inequality 
\[\log\rho(A)\ge\sup_{\mu=[\mu_{i,j}]\in\Omega(n)} \sum_{i,j\in[n]} \mu_{ij}\log \frac{a_{ij}\sum_{k=1}^n \mu_{ik}}{\mu_{ij}}.\] 
Assume that $A_1$ is an irreducible principle submatrix of $A$ such that $\rho(A)=\rho(A_1)$.  Then $\supp (A_1)\subset S\times S$ for some minimal nonempty subset of $[n]$.  Consider $\Omega(n,S\times S)$.  Now apply the theorem for the irreducible $A_1$ to deduce the theorem in this case.\end{proof}

\subsection{Ergodic control interpretation of the spectral radius}
\label{subsec-ergodic}
The variational characterization of the logarithm of the spectral radius, in Theorem~\ref{AGchar}, can be interpreted as follows in terms of ergodic control. We refer
the reader to~\cite{whittle86} for more background, and to~\cite{AG03} for a treatment adapted to the present setting. 

We associate to the matrix $a$ a one player stochastic game, with
state space $[n]$. The action space in state $i\in [n]$
is the simplex $\Pi_{i,n}:= \{p=(p_1,\dots,p_n)^T\in \Pi_n\mid a_{ij}=0 \implies p_j =0\}$, consisting of probability measures whose support is included
in the support of the $i$th line of $A$. In state $i$, if the player
selects action $p$, the next state becomes $j$ with probability $p_j$,
and the player receives a payment, given by
the Kullback-Leibler entropy
\[
\operatorname{KL}_i(p,a):= - \sum_{j\in [n]} p_j \log (p_j/a_{ij}) \enspace ,
\]
and the game is pursued in the same way, from the current state $j$.
The ergodic control problem consists in finding a strategy of the player
which maximizes the expected average payment per time unit.
It is known that if such a game is communicating, meaning that
for every states $i,j$, there is a strategy which ensures that the
probability of reaching $j$ in finite time starting from state $i$ is positive,
the value of the game is independent of the initial state. Here,
the communication assumption is equivalent to the irreducibility
of the matrix $A$. 

The value of these games has the following characterization. Recall that a
(feedback) {\em policy} is a map $\pi$ which associates to a state an admissible action in this 
state. So here, $\pi$ associates to $i$ a vector $\pi(i)\in \Pi_{i,n}$, and
we may identify $\pi$ to the stochastic matrix with rows $\pi(i)$, $i\in [n]$.
We denote by $M(\pi)$ the set of invariant measures of this matrix.

It is known, still under the communication assumption,
 that the value of the game, for any initial state,
coincides with the maximum over all policies $\pi$ and over
all invariant measures $\theta\in M(\pi)$ of the expectation
of the payment with respect to this measure, see 
 \cite[Proposition 7.2]{AG03}. When specialized
to the present setting, this formula shows that
\[
\log \rho(A) = \max_{\pi} \max_{\theta\in M(\pi)} 
\sum_{i\in[n]} \theta_i \operatorname{KL}_i(\pi(i),a) \enspace .
\]
Using the identification of $\pi$ to a stochastic matrix
$\nu\in \operatorname{Stoc}(c)$, 
this can be rewritten as 
\[\log\rho(A)=\max_{\nu=[\nu_{ij}]\in \operatorname{Stoc}(n), \mu=[\mu_{ij}]\in\Phi_n(\nu)}\sum_{i,j=1}^n \mu_{ij}\log \frac{a_{ij}}{\nu_{ij}},\]
which is equivalent
\eqref{AGchar1}. In the present case, concerning the spectral radius {{of}} a nonnegative matrix, characterizations of this nature go back to Donsker and Varadhan~\cite{DV75},
see also~\cite{AB17,akian_et_al:LIPIcs:2017:7026} for recent results
of this type. In particular, entropic payments of the type
considered here arise in the study of risk sensitive control
problems~\cite{AB17}. We next show that for nonnegative tensors,
the spectral radius still admits a characterization as the value
of an ergodic control problem.

\subsection{Entropic characterization of the spectral radius of a nonnegative tensor}\label{subsecmain}
We now extend the variational characterization \eqref{AGchar1} of the spectral
radius of a nonnegative matrix to the case of tensors.

In what follows we assume that $d\ge 3$ is an integer. 
For $\cT=[t_{i_1,\ldots,i_d}]\in\R_{ps,+}^{n^{\times d}}$
we denote by $\supp \cT$ the \emph{support} of the tensor $\cT$, i.e.,
\[
\supp \cT:= \{(i_1,\dots,i_d)\mid t_{i_1,\dots, i_d}>0 \} \enspace .
\]
A tensor $\mu=[\mu_{i_1,\ldots,i_d}]\in \R_{ps,+}^{n^{\times d}}$ is called
an {\em occupation measure} if
\begin{eqnarray}
 &&\sum_{i_1\in [n],\dots, i_d\in [n]}\mu_{i_1,i_2,\dots, i_d}=1,\notag\\
 &&\sum_{i_2,\dots,i_d\in [n]} 
 \mu_{j,i_2,\dots,i_d} =(d-1) \displaystyle\sum_{(i_1,i_3,\dots,i_d)\in [n]
}
 \mu_{i_1,j,i_3,\dots,i_d} 
 \qquad \forall j\in [n]\enspace .\label{tenocmeas}
\end{eqnarray}
Note that in view of the partial symmetry of $\mu$ the condition \eqref{tenocmeas}
is equivalent to
\[\sum_{(i_2,\dots,i_d)\in[n]} 
 \mu_{j,i_2,\dots,i_d} = \displaystyle\sum_{i_1,\dots,i_d\in [n],j\in (i_2,\dots,i_d) 
}
 \mu_{i_1,\dots,i_d} 
 \qquad \forall j\in [n]\enspace . \]
We denote by $\Omega(n^{\times (d-1)})\subset \R_{ps,+}^{n^{\times d}}$ the set of occupation measures.  For $\cT\in\R_{ps,+}^{n^{\times d}}$ we denote by $\Omega(n^{\times (d-1)},\supp \cT)\subseteq \Omega(n^{\times (d-1)})$ the set of occupation measures whose support is contained in $\supp \cT$.
The following is one of our main results.
\begin{theorem}[Entropic characterization of the spectral radius]\label{entropcharsrten}
The spectral radius of $\cT=[t_{i_1,\ldots,i_d}]\in \R_{ps.+}^{n^{\times d}}$ has the following characterization 
\begin{eqnarray}\label{entropcharsrten1}
\log \rho(\cT) &= 
\displaystyle 
\max_{\mu\in \Omega(n^{\times(d-1)})}
\sum_{i_1,\dots,i_d\in [n]}
\mu_{i_1,\dots, i_d} \log \Big(\frac{(\sum_{k_2,\dots, k_d}\mu_{i_1,k_2,\dots, k_d})t_{i_1,\dots, i_d}}{\mu_{i_1,\dots, i_d}}\Big) \enspace .
\label{e-occupation}
\end{eqnarray}
Assume that $\cT$ is weakly irreducible. Let $\uu>0$ be the unique positive eigenvector $\uu$ satisfying \eqref{poseigvec} and let $\w>\0$ be defined as in \eqref{defpropA}.  Let $\mu=[\mu_{i_1,\ldots,i_d}]\in \R_{ps,+}^{n^{\times d}}$ be the tensor given by
\begin{equation}\label{mumaxten}
\mu_{i_1,\ldots,i_d}=\frac{1}{\rho(\cT)} w_{i_1}u_{i_1}^{-(d-2)}t_{i_1,\ldots,i_d}u_{i_2}\cdots u_{i_d} \textrm{ for } i_1,\ldots,i_d\in[n].
\end{equation}
Then $\mu$ is an occupation measure whose support is $\supp \cT$. Furthermore,
\begin{equation}\label{logeqtenT}
\log\rho(\cT)=\sum_{i_1,\ldots,i_d\in [n]} \mu_{i_1,\dots, i_d} \log \Big(\frac{(\sum_{k_2,\dots, k_d}\mu_{i_1,k_2\dots ,k_d})t_{i_1,\dots, i_d}}{\mu_{i_1,\dots ,i_d}}\Big) \enspace .
\end{equation}
\end{theorem}
\begin{proof}  The proof of this theorem is analogous to the proof of Theorem \ref{AGchar} and we repeat briefly the analogous arguments.  Fix a weakly irreducible tensor $\cT=[t_{i_1,\ldots,i_d}]\in \R_{ps,+}^{n^{\times d}}$.  Assume that $\mu=[\mu_{i_1,\ldots,i_d}]\in \Omega(n^{\times(d-1)})$ has the same support as $\cT$.  Let $\nu=[\nu_{i_1,\ldots,i_d}]\in\R_{ps,+}^{n^{\times d}}$ be the following weakly irreducible tensor
\[\nu_{i_1,\ldots,i_d}=\frac{\mu_{i_1,\ldots,i_d}}{x_{i_1}}, \;x_{i_1}={\sum_{j_2,\ldots,j_d\in[n]}\mu_{i_1,j_2,\ldots,j_d}}, \quad i_1,\ldots,i_d\in[n].\]
Let $\x=(x_1,\ldots,x_n)\trans $  Since $\mu$ is a weakly irreducible tensor and an occupation measure if follows that $\x$ is a positive probability vector. 
Clearly $\nu\otimes^{d-1}\1_n=\1_n$.  Hence $\rho(\nu)=1$ and the corresponding eigenvector is $\1_n$.  Recall that $D\nu(\x)=(d-1) \nu\times \otimes^{d-2} \x$ (\eqref{partderT}).  Hence, the entries for the matrix $A(\mu)=(d-1)D\nu(\1_n)=[a_{ij}]\in\R_+^{n\times n}$, defined in \eqref{defATu} , are given by
\[a_{ij}=\frac{d-1}{x_i}\sum_{i_3,\cdots,i_d\in[n]}\mu_{i,j,i_3,\cdots,i_d}, \quad i,j\in[n].\]
Since $\mu$ is an occupation measure it follows that $\x\trans A(\mu)=\x\trans$
\[\sum_{i=1}^n x_i a_{ij}=(d-1)\sum_{i,i_3,\cdots,i_d\in[n]}\mu_{i,j,i_3,\cdots,i_d}=x_j \;\forall j\in[n].\]

Fix a tensor $\cB=[b_{i_1,\ldots,i_d}]\in\R_{ps}^{n^{\times d}}$ such that $b_{i_1,\ldots,i_d}=0$ if $\mu_{i_1,\ldots,i_d}=0$.  Let $\cC(t)=[\nu_{i_1,\ldots,i_d}e^{tb_{i_1,\ldots,i_d}}]\in\R_{ps,+}^{n^{\times d}}$ be the log-convex function on $\R$.  Clearly, each $\cC(t)$ is weakly irreducible. 
Hence $\log\rho(\cC(t))$ is a convex differentiable function on $\R$.  As in the proof of Lemma \ref{auxlemAG}, the variational formula \eqref{rhoTpert} implies that 
\[\log(\rho(\cC(t))'(t=0)= \sum_{i_1,\ldots,i_d\in[n]} \mu_{i_1,\ldots,i_d}b_{i_1,\ldots,i_d}.\]
The convexity of $\log\rho(\cC(t)$ and the equality $\log\rho(\cC(0))=\log\rho(\nu)=0$ yield that inequality 
\[\log\rho(\cC(1))\ge \sum_{i_1,\ldots,i_d\in[n]} \mu_{i_1,\ldots,i_d}b_{i_1,\ldots,i_d}.\]
Choose 
\[b_{i_1,\ldots,i_d}=\log\frac{t_{i_1,\ldots,i_d}}{\nu_{i_1,\ldots,i_d}} \textrm{ if } t_{i_1,\ldots,i_d}>0.\] 
Note that $\cC(1)=\cT$.  Hence 
\begin{equation}\label{baslogrineqten}
\log\rho(\cT)\ge \sum_{i_1,\ldots,i_d\in[n]}\mu_{i_1,\ldots,i_d}\log\frac{t_{i_1,\ldots,i_d}\sum_{j_2,\ldots,j_d\in[n]} \mu_{i_1,j_2,\ldots,j_d}}{\mu_{i_1,\ldots,i_d}}.
\end{equation}
The density argument yields that the above inequality holds for any $\mu\in\Omega(n^{\times(d-1)},\supp \cT)$.  Let $\mu\in\Omega(n^{\times(d-1)})$ and assume that $\supp \mu$ is not contained in $\supp \cT$.   Hence there exists a positive entry of $\mu$:  $\mu_{i_1,\ldots,i_d}$ such that $t_{i_1,\ldots,i_d}=0$.   Therefore 
\[\mu_{i_1,\ldots,i_d}\log\frac{t_{i_1,\ldots,i_d}\sum_{j_2,\ldots,j_d\in[n]} \mu_{i_1,j_2,\ldots,j_d}}{\mu_{i_1,\ldots,i_d}}=-\infty.\]
In this case \eqref{baslogrineqten} trivially holds.  These arguments show that $\log\rho(\cT)$ is not less that the right-hand side of \eqref{entropcharsrten1}.

Let $\mu\in\R_+^{n^{\times d}}$ be given by \eqref{mumaxten}.  As $\cT$ is partially symmetric it follows that $\mu$ is partially symmetric.  As $\uu$ is an eigenvector of $\cT$ corresponding to $\rho(\cT)$ we deduce that 
\[\sum_{i_2,\ldots,i_d\in[n]} \mu_{j,i_2,\ldots,i_d}=w_j u_{j}^{-(d-2)}u_j^{d-1}=w_j u_j \; \forall j\in [n].\]
As $\w\trans\uu=1$ it follows that $\mu$ is a probability tensor.  Let $A(\cT)$ be defined as in \eqref{defATu}.  Since $\w\trans A(\cT)=(d-1)\rho(\cT)\w\trans$, (the second equality in \eqref{defpropA}), it follows that $\mu$ satisfies \eqref{tenocmeas}. 
The equality \eqref{logeqtenT} is deduced in a similar way the equality in Lemma \ref{auxlemAG}.  

 Assume now that $\cT\in\R_{ps,+}^{n^{\times d}}$ is not weakly irreducible.  As in the proof of Theorem \ref{AGchar} it follows that the inequality \eqref{baslogrineqten} holds.  Suppose first that $\rho(\cT)=0$.  Then $\log\rho(\cT)=-\infty$.  Hence \eqref{entropcharsrten1} trivially holds.  Equivalently, for each $\mu\in\Omega(n^{\times (d-1)})$ there exists $i_1,\ldots,i_d\in[n]$ such that $\mu_{i_1,\ldots,i_d}>0$ and $t_{i_1,\ldots,i_d}=0$.
 
 Assume now that $\rho(\cT)>0$.    Let $\cJ_{n,d}\in\R_{ps,+}^{n^{\times d}}$ be a symmetric tensor whose all entries are $1$.  For a positive integer $l$ let $\cT_l=\cT+\frac{1}{l} \cJ_{n,d}$.  Then $\cT_l$ is weakly irreducible.  Our arguments yield that there exists positive occupation measure $\mu(l)=[\mu_{i_1,\ldots,i_d}(l)]\in\Omega(n^{\times (d-1)})$ such that 
 \[\log\rho(\cT_l)=\sum_{i_1,\dots,i_d\in [n]}
\mu_{i_1,\dots, i_d}(l) \log \Big(\frac{(\sum_{k_2,\dots, k_d}\mu_{i_1,k_2,\dots, k_d}(l))(t_{i_1,\dots, i_d}+\frac{1}{l})}{\mu_{i_1,\dots, i_d}(l)}\Big) \enspace .\]
As $\Omega(n^{\times (d-1)})$ is a compact set, there is a subsequence of $\{\mu(l)\},l\in \N$ which converges to the occupation measure $\mu\in\Omega(n^{\times (d-1)})$.  Hence
\[\log\rho(\cT)=\sum_{i_1,\dots,i_d\in [n]}
\mu_{i_1,\dots, i_d} \log \Big(\frac{(\sum_{k_2,\dots, k_d}\mu_{i_1,k_2,\dots, k_d})t_{i_1,\dots, i_d}}{\mu_{i_1,\dots, i_d}}\Big) \enspace .\]
Combine this equality with  the inequality \eqref{baslogrineqten} to deduce the theorem in this case.\end{proof}

\begin{rem}
  The function which is maximized in~\eqref{entropcharsrten1}
  is a relative entropy. This function is known to be concave; this follows
  from the fact that the perspective function of a convex function
  is convex, see~\cite[\S~3.2.6]{boyd}, and also~\cite{CS16}.
  \end{rem}

\begin{rem}
The log-convexity of the spectral radius of a nonnegative tensor, Corollary~\ref{logconvsr}, can be recovered from Theorem~\ref{entropcharsrten}, as formula~\eqref{entropcharsrten1} shows that the logarithm of the spectral radius, which is a maximum of linear functions of the logarithms of the entries of the tensor, 
is a convex function of these logarithms.
\end{rem}
\begin{rem}\label{rem-ergodiccontrol}
The ergodic control interpretation of the logarithm of the spectral radius, explained in \Cref{subsec-ergodic}, extends to the case of nonnegative tensors.
The set of actions of the player is still the finite set $[n]$. 
The set of actions in state $i$ consists of probability
measures $p=(p_{i,i_2,\dots,i_n})_{i_2,\dots,i_n}$ on the set $S_i:= \{(i_2,\dots,i_d)\in [n]^d\mid a_{i,i_2,\dots,i_n}>0$. If an action $p$ is selected, the next
state becomes $j$ with probability $\sum_{2\leq k\leq n,\;\; i_k = j} p_{i,i_2,\dots,i_n}$. Then, the player receives the payment 
\[
\operatorname{KL}_i(p,\cT)= -\sum_{(j_2,\dots,j_d)\in S_i} p_{i,i_2,\dots,i_d}\log(p_{i,i_2,\dots,i_d}/t_{i,i_2,\dots,i_d}) \enspace.
\]
We leave it to the reader to check, arguing as in \Cref{subsec-ergodic}, that 
the value of the associated ergodic game is independent of the initial
state as soon as $\cT$ is weakly irreducible, and that Formula~\eqref{logeqtenT}
allows us to identify $\log \rho(\cT)$ to the value of this game.
\end{rem}

 \section{Tropical spectral radius of nonnegative tensors}\label{sec:tropspecread}
Given $\x=(x_1,\ldots,x_n)\trans \in \C^{n}$, 
we set $\|\x\|_p:=(\sum_{i=1}^n |x_i|^p)^{\frac{1}{p}}$,
for $p\in (0,\infty]$.
 We start with a generalization of Karlin-Ost result \cite{KO85}.
 \begin{theorem}\label{genKO}  Let $\cT\in\R_{ps,+}^{n^{\times d}}$.  Then the function $s\mapsto \rho(\cT^{\circ s})^{\frac{1}{s}}$ is nonincreasing on $(0,\infty)$.
 \end{theorem}
 \begin{proof}  It is enough to show that
 \begin{equation}\label{spowin}
 \rho(\cT)\ge \rho(\cT^{\circ s})^{\frac{1}{s}} \textrm{ for } s>1.
 \end{equation}
  As in the proof of Lemma \ref{genkinin} we may assume
 that $\cT$ is weakly irreducible.  Let $\uu=(u_1,\ldots,u_n)\trans>\0$ be the eigenvector of $\cT$ satisfying \eqref{poseigvec}.
 Use the well known fact that $\|\x\|_p$ is a nonincreasing function of $p$ to deduce that
 \[\rho(\cT)u_i^{d-1}=\cT(\uu)_i\ge (\cT^{\circ s}(\uu^{\circ s})_i)^{\frac{1}{s}} \textrm{ for }  i\in [n] \Rightarrow \cT^{\circ s}(\uu^{\circ s})\le
  \rho(\cT)^s (\uu^{\circ s})^{d-1}.\]
 Use characterization \eqref{infmaxcharrho} to deduce \eqref{spowin}. \end{proof}

 Combine the above theorem with \eqref{genkinin1} to give a stronger version of \eqref{specradcircin}.
 \begin{equation}\label{specradcircinsv}
 \rho(\cT\circ\cS)\le \rho(\cT^{\circ \frac{1}{2}}\circ\cS^{\circ \frac{1}{2}})^2\le \rho(\cT)\rho(\cS), \quad \cS,\cT\in \R_{ps,+}^{n^{\times d}}.
 \end{equation}

We say that a nonzero nonnegative vector $u$ is a \emph{tropical eigenvector}
of the tensor $\cT\in\R_{ps,+}^{n^{\times d}}$, with the associated \emph{tropical eigenvalue} $\lambda$ if
\[
\max\{t_{i,i_2,\ldots,i_d}u_{i_2}\ldots u_{i_d},\;i_2,\ldots,i_d\in[n]\} = 
\lambda u_i^{d-1} \enspace,\qquad i\in [n] .
\]
The existence of a tropical eigenvector $u$ follows from a standard
application of Brouwer's theorem. Moreover, the number of distinct
tropical eigenvalues is bounded by $2^n-1$, this follows e.g.\ from~\cite[Th.~5.2.3]{nussbaumlemmens}.  The \emph{tropical spectral radius} of $\cT$, denoted by $\rho_{\trop}(\cT)$,
is defined as the maximal tropical eigenvalue of $\cT$.

We shall also consider the limit eigenvalue:
 \begin{equation}\label{tropspecrad}
 \rho_{\infty}(\cT):=\lim_{s\to\infty} \rho(\cT^{\circ s})^{\frac{1}{s}}.
 \end{equation}
 We first collect properties of the tropical spectral radius of $\cT\in\R_{ps,+}^{n^{\times d}}$, which follow from results of non-linear Perron-Frobenius theory~\cite{nuss86,GG04}.
 \begin{theorem}\label{tropspecradthm0}  Let $\cT\in \R_{ps,+}^{n^{\times d}}$.  Then
 \begin{equation}\label {tropspecradchar}
 \rho_{\trop}(\cT)=\inf_{\x=(x_1,\ldots,x_n)\trans>0} \max_{i\in[n]} \frac{\max\{t_{i,i_2,\ldots,i_d}x_{i_2}\ldots x_{i_d},\;i_2,\ldots,i_d\in[n]\}}{x_i^{d-1}}.
 \end{equation}
 There exists a tropical eigenvector $\bv=(v_1,\ldots,v_n) \trans \gneq \0$ corresponding to $\rho_{\trop}(\cT)$
 \begin{equation}\label{tropspeceig}
 \max\{t_{i,i_2,\ldots,i_d}v_{i_2}\ldots v_{i_d},\;i_2,\ldots,i_d\in[n]\}=\rho_{\trop}(\cT)v_i^{d-1} \textrm{ for } i\in [n].
 \end{equation}
 Assume that $\cT$ is irreducible.  Then every eigenvector $\bv$ satisfying \eqref{tropspeceig} is positive. 
Assume that $\cT$ is weakly irreducible. Then, there 
exists a positive eigenvector $\bv$ satisfying \eqref{tropspeceig},
and in the characterization
 \eqref{tropspecradchar}, the infimum can be replaced by the minimum.
 \end{theorem}
\begin{proof}
Formula~\eqref{tropspecradchar} follows from the Collatz-Wielandt characterization of the spectral radius of a non-linear map~\cite{nuss86}. The existence
of a positive eigenvector, if $\cT$ is weakly irreducible, follows
from the generalized Perron-Frobenius theorem~\cite[Theorem 2]{GG04}. 
When $\cT$ is irreducible, it is straightforward to check that any
nonnegative eigenvector must be positive.
\end{proof}
\begin{rem}
  The special case of~\Cref{tropspecradthm0} concerning irreducible
  tensors was proved in \cite[Th~3.3]{BUSH201964}.
  \end{rem}

\begin{proposition}\label{coro-lower}
Let $\cT\in \R_{ps,+}^{n^{\times d}}$.  Then for $s>0$  the following inequality hold.
\[
\rho(\cT^{\circ s})^{\frac{1}{s}}\geq \rho_{\trop}(\cT) \enspace .
\]
\end{proposition}
\begin{proof}
 {{The characterization \eqref{infmaxcharrho} yields the Collatz-Wielandt characterization of $\rho(\cT^{\circ s})^{\frac{1}{s}}$ for a positive $s>0$:
 \[\rho(\cT^{\circ s})^{\frac{1}{s}}= \inf_{\x=(x_1,\ldots,x_n)\trans>\0}\max_{i\in[n]} \frac{(\cT^{\circ s}(\x^{\circ s})_i)^{\frac{1}{s}}}{x_i^{d-1}}.\]
 Clearly,
\begin{eqnarray*}\label{basintropchar}
\max\{t_{i,i_2,\ldots,i_d}x_{i_2}\ldots x_{i_d},\;i_2,\ldots,i_d\in[n]\}\le (\cT^{\circ s}(\x^{\circ s})_i)^{\frac{1}{s}}, \quad i\in[n]. 
\end{eqnarray*} 
Compare the Collatz-Wielandt characterizations of $\rho(\cT^{\circ s})^{\frac{1}{s}}$  and $\rho_{\trop}(\cT)$, given by \eqref{tropspecradchar}, to deduce the lemma.}}
 \end{proof}

 \begin{theorem}\label{tropspecradthm}  Let $\cT\in \R_{ps,+}^{n^{\times d}}$.  Then
\[
\rho_{\trop}(\cT) = \rho_\infty(\cT) \enspace .
\]
 \end{theorem}
 \begin{proof}
   {{In the inequality given by Proposition \eqref{coro-lower}, let $s\to\infty$ to obtain $\rho_{\trop}(\cT)\le \rho_{\infty}(\cT)$.

Assume that $\x>\0$.  The Collatz-Wielandt characterization of $\rho(\cT^{\circ s})^{\frac{1}{s}}$ yields that $\rho(\cT^{\circ s})^{\frac{1}{s}}\le \max_{i\in[n]} \frac{(\cT^{\circ s}(\x^{\circ s})_i)^{\frac{1}{s}}}{x_i^{d-1}}$.
 Letting $s\to\infty$ we deduce that
 \begin{eqnarray*}\label{uptorpbound}
 \rho_{\infty}(\cT)\le  \max_{i\in[n]} \frac{\max\{t_{i,i_2,\ldots,i_d}x_{i_2}\ldots x_{i_d}.\;i_2,\ldots,i_d\in[n]\}}{x_i^{d-1}},
 \end{eqnarray*}
Use~\eqref{tropspecradchar} to deduce
$\rho_\infty(\cT)\leq \rho_\trop(\cT) \enspace$.}}

 \end{proof}

 Given a tensor $\cT=[t_{i_1,\ldots,i_d}]\in\R_{ps,+}^{n^{\times d}}$ we define 
 the tensor \emph{pattern} of $\cT$, 
$\pat \cT=[t'_{i_1,\ldots,i_d}]\in \R_{ps,+}^{n^{\times d}}$,
to be the following $0-1$ tensor:  $t'_{i_1,\ldots,i_d}=1$ if $t_{i_1,\ldots,i_d}>0$ and otherwise
 $t'_{i_1,\ldots,i_d}=t_{i_1,\ldots,i_d}=0$.
 \begin{theorem}\label{rhotropin}  Let $\cT,\cE\in \R_{ps,+}^{n^{\times d}}$.  Then the following inequalities hold.
 \begin{eqnarray}\label{rhotropin1}
 \rho(\cT\circ\cE)\le \rho(\cT)\rho_{\trop}(\cE),\\
 \rho(\cE)\le \rho(\pat \cE)\rho_{\trop}(\cE),  \label{rhotropin2}
 \end{eqnarray}
in particular
\begin{eqnarray}
 \rho(\cE)\leq {n^{d-1}}\rho_{\trop}(\cE) \enspace .  \label{rhotropin3}
\end{eqnarray}
 \end{theorem}
 \begin{proof}
 The inequality \eqref{genkinin1} is equivalent to
 \begin{equation}\label{genkinin1A}
 \rho(\cT\circ \cE)\le \rho(\cT^{\circ p})^{\frac{1}{p}}  \rho(\cE^{\circ q})^{\frac{1}{q}}, \quad p,q >1,\; \frac{1}{p}+\frac{1}{q}=1.
 \end{equation}
 Let $p\searrow 1$ to deduce \eqref{rhotropin1}.  Let $\cT=\pat \cE$ to deduce \eqref{rhotropin2} from \eqref{rhotropin1}. 
Finally, it follows from the Collatz-Wielandt characterization of the spectral radius, Eqn~\eqref{infmaxcharrho},
that the map which associates to a tensor its spectral radius is a nondecreasing function of the entries of the tensor. Since $\pat \cE\leq \cJ$, where $\cJ$ is the tensor with identically $1$ entries, it follows that $\rho(\pat\cE)\leq \rho(\pat\cJ)=n^{d-1}$. \end{proof}

 The inequality \eqref{rhotropin2} is a generalization of the inequality for matrices given in \cite{Fri86}.

Combining \Cref{coro-lower} and \Cref{rhotropin3}, we obtain the following
sandwitch inequality.
\begin{corollary}[Generalization of the Cauchy-Birkhoff-Fujiwara bound to nonnegative tensors]
\label{cbf}
Let $\cE\in \R_{ps,+}^{n^{\times d}}$.  Then the following inequalities hold.
\begin{equation}
\rho_{\trop}(\cE) \leq \rho(\cE) \leq n^{d-1} \rho_{\trop}(\cE) \enspace .
\label{e-metric}
\end{equation}
\end{corollary}
These inequalities are tight.
Corollary~\eqref{cbf}
is similar in essence to classical bounds
for the maximal modulus of a root of a complex polynomial
$p(z)=\sum_{k=0}^n a_k z^k$ of degree $n$. Let 
\[
\alpha_{\trop}:= \max_{k\in [n-1]} \Big(\frac{|a_k|}{|a_n|}\Big)^{\frac{1}{n-k}} \enspace.
\]
This quantity can be interpreted as the exponential of the greatest ``tropical
root'' of the polynomial $p$, it is an extremal slope of a Newton polytope
associated with $p$, see~\cite{Ostrowski1}, and~\cite{logmajorization2013}
for a recent treatment inspired by tropical geometry. 
It is known that if $\zeta$ is a root of maximal modulus,
\begin{equation}
\frac{1}{n} \alpha_{\trop} \leq |\zeta|\leq 2 \alpha_{\trop}  \enspace .
\label{e-cp}
\end{equation}
Indeed, the first inequality is due to Birkhoff~\cite{birkhoff}, whereas the last one
is a homogeneous version of the classical Cauchy bound, due to Fujiwara~\cite{fujiwara}. 
The inequalities~\eqref{e-cp} should be compared with~\eqref{e-metric}, they
show that the modulus of a classical root is bounded from above and from
below by its tropical analogue, up to combinatorial constants.
The inequalities~\eqref{e-cp}  and~\eqref{e-metric}, which relate classical algebraic objects with their tropical analogues, are of current interest in tropical geometry, being related with the metric theory of amoebas~\cite{AVENDANO201845}. 

We next provide a combinatorial expression of the tropical spectral
radius of a nonnegative tensor. For comparison, it is convenient
to recall the expression of the tropical spectral radius of a nonnegative square matrix $T=[t_{i,j}]\in\R^{n\times n}$, see~\cite{bcoq,butkovic}.

 Let $T=[t_{ij}]\in\R_+^{n\times n}$. With each cycle $\gamma\in\Sigma_n$ we associate a weighted average of the entries of $T$ along $\gamma$.  (See the beginning of \S\ref{sec:entropsr}.)
 \begin{equation}\label{defwcycle}
 w(\gamma,T)=(\prod_{j\in [k]} t_{i_ji_{j+1}})^{\frac{1}{k}}, \quad i_{k+1}\equiv i_1.
 \end{equation}
It is known that for a tropical matrix,
\begin{equation}\label{matropsrchar}
 \rho_{\trop}(T)=\max_{\gamma\in\Sigma_n} w(\gamma,T) \enspace,
\end{equation}
see~\cite{bcoq,butkovic}.
Moreover, Friedland showed in \cite{Fri86} that $\rho_\infty(T)$
is given by the same expression.

We associate with $T$ the digraph $\digraph(T)=([n],\diedges)$, where $[n]$ is the set of vertices and $\diedges\subset [n]\times [n]$ is the set of directed edges.
There is a directed edge $(i,j)$ from the vertex $i$ to the vertex $j$ if $t_{ij}>0$.    Denote by $\Sigma(T)$ the set of all dicycles $\gamma$ in $\digraph(T)$.  Note that in \eqref{matropsrchar} we can restrict the maximum  over $\gamma\in \Sigma(T)$.

We now extend the characterization~\eqref{matropsrchar} to the case of tensors.
Let $k\in\N$.  Denote by $\kdigraph_{n,k}=(V,\Arcs_k)$ a complete $k$-multi digraph on $V=[n]$ vertices.  That is, each diedge $(i,j)$ in $\kdigraph_{n,k}$ appears exactly
 $k$ times.  Let $\digraph=(V',\Arcs'), V'\subset V, \Arcs'\subset \Arcs_k$ be a subgraph of $\kdigraph_{n,k}$.
 Then $A(\digraph)=[a_{uv}], u,v\in V'$ is called the adjacency matrix of $\digraph$ if $a_{uv}$ is the number of diedges $(u,v)$ in $\digraph$.  $\digraph$
 is called a $k$-cycle if the following conditions hold. First, $\digraph$ is strongly connected, i.e.\ $A(\digraph)$ is an irreducible matrix.
 Second, the out-degree of each vertex $v\in V'$ is $k$.  So for each $v\in V'$ we denote by $(v,j_2(v,\digraph)),\ldots,(v,j_{k+1}(v,\digraph))$
 all diedges from the vertex $v$ in the cycle $\digraph$.  We assume here
 \begin{equation}\label{vertoutv}
 1\le j_2(v,\digraph)\le \ldots\le j_{k+1}(v,\digraph)\le n.
 \end{equation}
 Denote by $\Sigma_{n,k}$ the set of $k$-cycles in $\kdigraph_{n,k}$.  We denote a $k$-cycle by
 $\gamma\in\Sigma_{n,k}$.  Note that $1$-cycle is a cycle defined as above.  Assume that $\gamma=(V(\gamma),\Arcs(\gamma))\in\Sigma_{n,k}$.  Let $A(\gamma)$ be the adjacency matrix of $\gamma$.  Denote $\1_{V(\gamma)}=(1,\ldots,1)\trans\in \R^{|V(\gamma)|}$.  The assumption that $\gamma$ is $k$-cycle implies that
 $A(\gamma)\1_{V(\gamma)}=k\1_{V(\gamma)}$.  Since $A(\gamma)$ is irreducible, there exists a unique probability vector $\uu(\gamma)$ that is a left eigenvector of $A(\gamma)$:
 \begin{equation}\label{defleAgam}
  A(\gamma)\trans \uu(\gamma)=k\uu(\gamma), \quad \0<\uu(\gamma)=(u(\gamma)_v),\;v\in V(\gamma),\;\sum_{v\in V(\gamma)} u(\gamma)_v=1 \enspace.
 \end{equation}

 Let $\cF=[f_{i_1,\ldots,i_d}]\in \R_{ps,+}^{n^{\times d}}$.   With each $d-1$ cycle
 $\gamma$ associate the following weighted average of the entries of $\cF$ along $\gamma$.
  \begin{equation}\label{defwcycleten}
 w(\gamma,\cF)=(\prod_{v\in V(\gamma)} f_{v,j_2(v,\gamma),\ldots,j_d(v,\gamma)})^{u(\gamma)_v}.
 \end{equation}
 \begin{theorem}\label{chartropeigten}  Let $\cF=[f_{i_1,\ldots,i_d}]\in\R_{ps.+}^{n^{\times d}}$.  Then
 \begin{equation}\label{chartropeigten1}
 \rho_{\trop}(\cF)=\max_{\gamma\in\Sigma_{n,d-1}} w(\gamma,\cF).
 \end{equation}
 \end{theorem}
\begin{proof}
Assume first that $\rho_{\trop}(\cF)>0$.   Let $V'\subset [n]$ will be the smallest subset of indices for which the coordinate $v_i$ of the tropical eigenvector $\bv$ satisfying \eqref{tropspeceig} with the following restriction. For each $i\in V'$ $v_i>0$ and the corresponding maximum in \eqref{tropspeceig} can be taken only on $i_2,\ldots,i_d\in V'$.

For each $k\in V'$,
we choose indices
$i_2=j_2(k),\ldots,i_d=j_d(k)\in V'$
achieving the maximum in \eqref{tropspeceig}, so that
\begin{equation}
\label{eq-saturate}
t_{i,j_2(i),\ldots,j_d(i)} v_{j_2(i)}\ldots v_{j_d(i)}
 = \rho_{\trop}(\cT) v_i^{d-1}, \qquad i\in V'.
\end{equation}

This defines a digraph $\gamma$.  The minimality of $V'$ implies that $\gamma\in\Sigma_{n,d-1}$ is a $(d-1)$-dicycle.  Let $\uu(\gamma)$ be defined from $\gamma$,
as in~\eqref{defleAgam}.
We now raise each term of the $i$th equality~\eqref{eq-saturate} to the power
$u(\gamma)_i$, 
\begin{equation}\label{e-saturate2}
(t_{i,j_2(i),\ldots,j_d(i)} v_{j_2(i)}\ldots v_{j_d(i)} )^{u(\gamma)_i}
 = (\rho_{\trop}(\cT)v_i^{d-1})^{u(\gamma)_i}, \qquad i \in V' \enspace .
\end{equation}
We next multiply all the equalities~\eqref{e-saturate2}, and observe
that, thanks to~\eqref{defleAgam}, the terms involving powers of $v$ can be canceled, showing that
\[ \rho_{\trop}(\cF)= w(\gamma,\cF)
\leq \max_{\gamma'\in\Sigma_{n,d-1}} w(\gamma',\cF)
\enspace.\]
We show the reverse inequality. Given $\gamma\in \Sigma_{n,d-1}$ let $\cF(\gamma)=[f(\gamma)_{i_1,\ldots,i_d}]$ be the following symmetric tensor in the last $d-1$
 indices: $f(\gamma)_{i,i_2,\ldots,i_d}=f_{i,i_2,\ldots,i_d}$ if an only if $i\in V(\gamma)$ and  $(i_2,\ldots,i_d)$ are permutations of $(j_2(i,\gamma),
 \ldots,j_d(i,\gamma)$.  Otherwise $f_{i,i_2,\ldots,i_d}=0$.  Clearly, $\cF(\gamma)\le \cF$.  Hence $\rho_{\trop}(\cF(\gamma))\le \rho_{\trop}(\cF)$.
 Arguing as in the first part of the proof, we show that $\rho_{\trop}(\cF(\gamma))=w(\gamma,\cF)$.  Hence we have characterization \eqref{chartropeigten1}.

 The above arguments show that $\rho_{\trop}(\cF)=0$ if and only if each $\rho_{\trop}(\cF(\gamma))=0$ for each $\gamma\in\Sigma_{n,k}$.  
\end{proof}
Since the tensor $\cT$ is supposed to be symmetric in the indices $i_2,\dots,i_d$, 
for computational purposes,
we will use
a concise encoding of the support, $\bar{\support}(\cT)\subset
\supp \cT$, so that $\bar{\support}(\cT)$ contains precisely one
element $(i_1,i_2,\dots,i_d)$ in each symmetry class $\{(i_1,\sigma(i_2),\dots,\sigma(i_d))\mid \sigma\in S_{d-1}\}$, where $S_{d-1}$ denotes the symmetric
group on $d-1$ symbols. 
Observe that in the tropical eigenvalue problem~\eqref{tropspeceig}, we may restrict the maximization to sequences
$(i_1,\dots,i_d)$ belonging to $\bar{\support}(\cT)$.

The following is an immediate corollary of the Collatz-Wielandt formula~\eqref{tropspecradchar}.
\begin{corollary}\label{cor-lp}
Let $\cT=[t_{i_1,\ldots,i_d}]\in\R_{ps.+}^{n^{\times d}}$.  Then, $\log\rho_{\trop}(\cT)$ coincides
with the value of the following linear program
\begin{eqnarray}
\min \lambda,& \lambda \in \R, \qquad u\in \R^n &\nonumber \\ 
\log t_{i_1,\dots,i_d} + u_{i_1} + \dots +u_{i_d} &\leq \lambda + (d-1) u_{i_1} ,
\qquad \forall (i_1,\dots,i_d) \in \bar{\support}(\cT) \enspace .
\label{e-ineqlp}
\end{eqnarray}
In particular, $\log\rho_{\trop}(\cT)$ can be computed in polynomial time
in the Turing model of computation, assuming that the input consists
of the set $\bar{\support}(\cT)$ and 
of numbers $\log t_{i_1,\dots,i_d} \in \mathbb{Q}$ for $(i_1,\dots,i_d)\in 
\bar{\support}(\cT)$.\hfill\qed
\end{corollary}
It follows from the strong duality theorem in linear programming
that the value of the linear program in Corollary~\ref{cor-lp}
coincides with the one of its dual. By computing the dual linear
program, we obtain the following consequence of Corollary~\ref{cor-lp},
in which $\mu_{i_1,\dots,i_d}$ denotes the Lagrange multiplier
of the inequality constraint~\eqref{e-ineqlp}.
\begin{corollary}\label{cor-lp2}
Let $\cT=[t_{i_1,\ldots,i_d}]\in\R_{ps.+}^{n^{\times d}}$.  Then, $\log\rho_{\trop}(\cT)$ coincides
with the value of the following linear program
\begin{eqnarray*}
\max \sum_{(i_1,\dots,i_d)\in \bar{\support}(\cT)}
\mu_{i_1,\dots,i_d} \log t_{i_1,\dots i_d,} \nonumber\\
\mu_{i_1,\dots,i_d}\geq 0, & \text{for } (i_1,\dots,i_d)\in 
 \bar{\support}(\cT)\nonumber\\
 \sum_{(i_1,\dots,i_d)\in \bar{\support}(\cT)\atop i_1=j} 
 \mu_{i_1,\dots,i_d} &= \displaystyle\sum_{(i_1,\dots,i_d)\in \bar{\support}(\cT)\atop j=i_2,\dots,i_d 
}
 \mu_{i_1,\dots,i_d} 
 \qquad \forall j\in [n]\enspace ,\\
\sum_{(i_1,\dots,i_d)\in \bar{\support}(\cT)}
 \mu_{i_1,\dots,i_d} &=1 
\end{eqnarray*}
\hfill\qed
\end{corollary}
The next corollary follows from~\eqref{chartropeigten1}.
 \begin{corollary}\label{logconvtrop}  Let $D\subset \R^m$ be a convex set.  Assume that $\cT:D\to \R_+^{n^{\times d}}$ is logconvex.
 Then $\rho_{\trop}(\cT):D\to\R_+$ is logconvex.
\hfill\qed
 \end{corollary}

\begin{rem}
It follows from the above linear programming formulations
that $\log\rho_{\trop}$ coincides
with the value of an ergodic Markov decision process (MDP), i.e.,
a one player stochastic game with mean payoff~\cite{whittle86},
in which the state space is $[n]$. This game appears
to be ``a degeneration'' of the game with entropic 
payment in Remark~\ref{rem-ergodiccontrol}, in which now, the action spaces
become finite.  
Let us describe
this MDP. In a given state $j$, the set
of actions consists of $\{(j,i_2,\dots,i_d)\in \bar{\support}(\cT)\}$,
the player receives the payment $\log t_{j,i_2,\dots ,i_d}$, 
the next state become $k$ with probability $|\{2\leq \ell\leq d\mid 
i_\ell =k\}|/(d-1)$, and
$\log \rho_{\trop}$ represents the best mean payoff per time unit.
Then, the digraphs $\gamma$ arising in the combinatorial
characterization~\eqref{chartropeigten1} correspond to feedback policies,
and this characterization shows that $\log \rho_{\trop}$ is the
supremum of the ergodic payments attached to the different
feedback policies, a general property of ergodic Markov decision processes~\cite[Prop.~7.2]{AG03}.
The dual variables $\mu_{i_1,\dots,i_d}$ represent
an {\em occupation measure}, giving the frequency
at which an action $(i_1,\dots,i_d)$ is played.
\end{rem}
\begin{example}
  Take $n=2$ and $d=4$. Then, the general tropical eigenproblem
  can be written as follows, 
\begin{eqnarray}
  \lambda v_1^3 & =& \max(t_{1111}v_1^3,\,
   t_{1112}v_1^2v_2, \, t_{1122}v_1v_2^2,\,
  t_{1222}v_2^3)\label{e-eq-t1}
\\
\lambda v_2^3 &=&
\max(t_{2111}v_1^3,\,
  t_{2112}v_1^2v_2, \, t_{2122}v_1v_2^2,\,
  t_{2222}v_2^3)
\enspace .\label{e-eq-t2}
\end{eqnarray}
We suppose here, without loss of generality, that the tensor
is symmetric in the last $3$ indices.
We next list the different cycles appearing in the representation
of \Cref{chartropeigten}.

First, there are two $3$-cycles with vertex set of cardinality $1$ (multiloops).
The first of these $3-$cycles arises by considering the multigraph $\gamma^1_1$, with vertex set $V_1^1=\{1\}$ and edges multiset $E^1_1=\{(1,1), (1,1),(1,1)\}$, meaning that edge $(1,1)$ is repeated $3$ times.
Then, the adjacency matrix $A(\gamma^1_1)$ is the $1\times 1$ matrix equal to $(3)$, the associated invariant measure $u(\gamma^1_1)$ is the one-dimensional vector $(1)$, leading to the weight
  \begin{equation}
       w(\gamma^1_1)  =t_{1111} \enspace .\label{e-gamma1}
  \end{equation}
  Similarly, there is another $3$-cycle $\gamma^1_2$ with vertex
  set of cardinality $1$, namely $V^1_2=\{2\}$,
  with edge multiset $E^1_2=\{(2,2),(2,2),(2,2)\}$,
  leading to the weight
    \begin{equation}
       w(\gamma^1_2)  =t_{2222} \enspace .\label{e-gamma2new}
    \end{equation}
    Less trivial examples arise when considering $3$-cycles
    with vertex set of cardinality $2$, like $\gamma_1^2$,
    with vertex set $V^2_1=\{1,2\}$ and edge multiset
    $E^2_1=\{(1,1),(1,1), (1,2),
  (2,1)$, $(2,1),(2,1)\}$. The associated adjacency matrix is
  \[
  A(\gamma_1^2)= \left(\begin{array}{cc} 2 & 1 \\
    3 & 0 \end{array}\right)
  \]
  with invariant measure $u(\gamma^2_1)= (3/4,1/4)$,
  leading to the weight
  \begin{equation}
    w(\gamma_1^2)= t_{1112}^{3/4}t_{2111}^{1/4} \enspace .\label{e-gamma2}
  \end{equation}
  This suggests the following rule to construct all the $3$-cycles
  with a vertex set of cardinality $2$: 
  --  in state $1$, select one of the
  terms at the right hand side of~\eqref{e-eq-t1}
  except the first term $t_{1111}v_1^3$ (which would lead to a multiloop),
  --- in state $2$, select one of the
  terms at the right hand side of~\eqref{e-eq-t2},
  except the first term $t_{2222}v_2^3$;
  Then, the exponents of the terms which are selected
  determine the edges of the multigraph.
  For instance, selecting
  $t_{1112}v_1^2v_2$ in state $1$, and
  $t_{2111}v_1^3$ in state $2$, yields
  the exponent vectors $(2,1)$ and $(3,0)$,
  corresponding precisely to the rows of the
  matrix $A(\gamma_1^2)$. It follows
  that there are precisely $9$ $3$-cycles not reduced
  to multiloops, corresponding
  to these different selections. We next
  list the adjacency matrices of all these
  $3$-cycles, and the associated weights,
  \begin{eqnarray*}
    \begin{array}{c|c|c|c|c|c|c|c|c}
            \left(\begin{smallmatrix}
        2 & 1 \\
         3 & 0      \end{smallmatrix}\right) &
      \left(\begin{smallmatrix}
        2 & 1 \\
         2 & 1      \end{smallmatrix}\right) &
            \left(\begin{smallmatrix}
        2 & 1 \\
         1 & 2            \end{smallmatrix}\right) &
                        \left(\begin{smallmatrix}
        1 & 2 \\
        3 & 0                        \end{smallmatrix}\right) &
                        \left(\begin{smallmatrix}
        1 & 2 \\
         2 & 1                        \end{smallmatrix}\right) &
                        \left(\begin{smallmatrix}
        1 & 2 \\
         1 & 2                        \end{smallmatrix}\right) &
                                                \left(\begin{smallmatrix}
        0 & 3 \\
        3 & 0                        \end{smallmatrix}\right) &
                        \left(\begin{smallmatrix}
        0 & 3 \\
         2 & 1                        \end{smallmatrix}\right) &
                        \left(\begin{smallmatrix}
        0 & 3 \\
        1 & 2                        \end{smallmatrix}\right)\\ &&&&&&&\\ 
                        \scriptstyle t_{1112}^{3/4}t_{2111}^{1/4}& \scriptstyle t_{1112}^{2/3}t_{2112}^{1/3}&
                        \scriptstyle t_{1112}^{1/2}t_{2122}^{1/2}&
                        \scriptstyle t_{1122}^{3/5}t_{2111}^{2/5}&
                        \scriptstyle t_{1122}^{1/2}t_{2112}^{1/2}&
                        \scriptstyle t_{1122}^{1/3}t_{2122}^{2/3}&
                        \scriptstyle t_{1222}^{1/2}t_{2111}^{1/2}&
                        \scriptstyle t_{1222}^{2/5}t_{2112}^{3/5}&
                        \scriptstyle t_{1222}^{1/4}t_{2122}^{3/4}
      \end{array}
    \end{eqnarray*}
  \Cref{chartropeigten} shows that the tropical spectral
  radius of $\cT$ is given by the maximum of the weights~\eqref{e-gamma1},
  \eqref{e-gamma2new} and of the weights appearing in the latter table.
\end{example}
\begin{rem}%
Corollary~\ref{cor-lp} leads to a polynomial time algorithm
to compute the spectral radius of a tropical tensor. 
The reduction to ergodic Markov decision
processes allows us, more generally, to
apply any algorithm developed in this setting, including policy
iteration~\cite{whittle86}. For huge scale instances, iterative 
power type algorithms
may be more suitable. One may use
the relative value iteration~\cite{white}.
This algorithm does converge if the optimal strategies
satisfy a certain cyclicity condition
(Corollary~5.9 and Theorem 6.6 of \cite{AG03}).
One can also use the projective Krasnoselskii-Mann iteration
described in~\cite{Gstott}, section 5,
which is a power type algorithm with damping
that converges without any condition of cyclicity.
\end{rem}
 \section{Inequalities for spectral norms of nonnegative tensors}\label{sec:specnorm}
 Let $\cT\in\R^{\m}$.  Recall the definition of the spectral norm of $\cT$, see~\cite{FL17},
\begin{equation}\label{defspecnrm}
 \|\cT\|_{\infty}=\max\{|\cT\times(\otimes_{j=1}^d \x_j)|, \;\|\x_j\|=1, j\in[d]\},
 \end{equation}
We now show that the spectral radius of a $d$-mode $n$-equidimensional tensor is bounded above by its spectral norm times the factor {{$n^{(d-2)/2}$}}. 
 \begin{lemma}\label{specneqrad}  Let $\cT\in\C^{n^{\times d}}_{ps}$.  
Then
 \begin{equation}\label{specneqrad1}
 \rho(\cT)\le \|\cT\|_{\infty}n^{(d-2)/2}.
 \end{equation}
 \end{lemma}
 \begin{proof}  Assume that $\cT\times (\otimes^{d-1}\x)=\lambda \x^{(d-1)}$ and $|\lambda|=\rho(\cT)$. 
 Normalize $\x$ by $\|\x^{(d-1)}\|=1$.  
Let $\y=\bar\x^{(d-1)}$, the complex conjugate of $\x^{(d-1)}$, so that $\|\y\|=1$.  Then $|\cT\times(\y\otimes(\otimes^{d-1}\x)|= |\lambda|$.  
Therefore $|\lambda|\le \|\cT\|_{\infty}\|\y\|\|\x\|^{d-1}=\|\cT\|_\infty \|\x\|^{d-1}$.  Use H\"older's inequality
to deduce that $\|\x\|^{d-1}
\leq 
n^{\frac{d-2}{2}}$.
 \end{proof}
\begin{rem}
In Lemma~\ref{specneqrad}, we assumed that $\cT\in \C^{n^{\times d}}_{ps}$ 
since we only considered the eigenproblem for partially
symmetric tensors. The inequality~\eqref{specneqrad1} carries
over to any $\cT\in \cT\in\C^{n^{\times d}}$, understanding that $\rho(\cT)$
only depends of the partially symmetric part of $\cT$.
\end{rem}

 Clearly, for a nonnegative tensor 
 \begin{equation}\label{infnrmnnc}
 \|\cT\|_{\infty}=\max\{\cT\times(\otimes_{j=1}^d \x_j), \;\|\x_j\|=1, \x_j\ge \0, j\in[d]\} \textrm{ for } \cT\in\R_+^{\m}.
\end{equation}

The following theorem gives inequalities on the spectral norms of tensors.  Some of them are well known, and we bring them for completeness.
Some other generalize the results of \S\ref{sec:varspecrad}, \S\ref{sec:logconv} and \S\ref{sec:tropspecread} to the spectral norm of nonnegative tensors.

\begin{theorem}\label{logconspecnrm}   $\;$
\begin{enumerate}
\item For $\cT=[t_{i_1,\ldots,i_d}]\in\R^{\m}$ let $|\cT|$ be the tensor in $\R^{\m}$ whose entries are the absolute value of the entries $\cT$: $|\cT|=[|t_{i_1,\ldots,i_d}|]$.
Then
\begin{equation}\label{absdomin}
\|\cT\|_{\infty}\le \||\cT|\|_{\infty}.
\end{equation}
Furthermore
\begin{equation}\label{lowbndspecnrm}
\max_{i_1\in[m_1],\cdots,i_d\in[m_d]}|t_{i_1,\ldots,i_d}|\le \|\cT\|_{\infty}.
\end{equation}
\item Let $\n=(n_1,\ldots,n_k)\in\N^k, \m=(m_1,\ldots,m_d)\in\N^d$.  Assume that $\cE\in\R^{\n}, \cT\in\R^{\m}$.  Then
\begin{equation}\label{specnrmtenprod}
\|\cE\otimes \cT\|_{\infty}=\|\cE\|_{\infty}\|\cT\|_{\infty}
\end{equation}
\item Let us still make the assumptions of \emph{2}.  Suppose furthermore that $\cE$ is a subtensor of $\cT$.  
Then
\begin{equation}\label{subtenin}
\|\cE\|_{\infty}\le\|\cT\|_{\infty}.
\end{equation}
\item Assume that $\cE,\cT\in\R^{\m}_+$.  Then
\begin{equation}\label{hadprodin}
\|\cE\circ\cT\|_{\infty}\le \|\cE\|_{\infty}\|\cT\|_{\infty}.
\end{equation}
\item
Let $D\subset \R^m$ be a convex set.  Assume that $\cT:D\to \R_+^{\m}$ is logconvex.
 Then $\|\cT\|_{\infty}:D\to\R_+$ is logconvex.
\item Let $\cF,\cG\in\R_+^{\m}$.  Then
 \begin{equation}\label{genkinspecin}
 \|\cF^{\circ \alpha}\circ\cG^{\circ\beta}\|_{\infty}\le \|\cF\|_{\infty}^\alpha \|\cG\|_{\infty}^\beta,\; \alpha,\beta> 0, \alpha+\beta=1.
 \end{equation}
 \item Assume that $\cT=[t_{j_1,\ldots,j_d}]\in \R^{\m}_+$.  Then the function $\|\cT^{\circ s}\|_{\infty}^{\frac{1}{s}}$ is a decreasing function on $(0,\infty)$.
 Furthermore 
 \begin{equation}\label{limspecnrm}
 \lim_{s\to\infty} \|\cT^{\circ s}\|_{\infty}^{\frac{1}{s}}=\max_{j_1\in[m_1],\ldots,j_d\in[m_k]} t_{j_1,\ldots,j_d},
 \end{equation}
 is the standard $\ell_{\infty}$ norm of $\cT$ viewed as a vector.
 \item  Assume that $\cE,\cF\in\R^{\m}_+$.  Then
 \begin{eqnarray}\label{increspropspectnrm}
 &&\|\cE\|_{\infty}\le \|\cT\|_{\infty} \textrm{ if } \cE\le \cF,\\
 \label{ineqETtropnorm}
 &&\|\cT\circ\cE\|_{\infty}\le \|\cT\|_{\infty}\|\cE\|_{\ell_{\infty}}, \;
 \|\cE\|_{\infty}\le \|\pat \cE\|_{\infty}\|\cE\|_{\ell_{\infty}}.  
 \end{eqnarray}
 \end{enumerate}
\end{theorem}
\begin{proof}   \emph{1}.  As $|\cT\times(\otimes_{j=1}^d \x_j)|\le |\cT|\times(\otimes_{j=1}^d|\x_j|)$, the maximal characterization of $\|\cT\|_{\infty}$ and $\||\cT|\|_{\infty}$ yields
\eqref{absdomin}.  By choosing $\x_j$ to be a canonical basis vector $(\delta_{1i_j},\ldots,\delta_{m_ji_j})\trans\in\F^{m_j}$ we deduce that $|\cT\times (\otimes_{j=1}^d \x_j))=
|t_{i_1,\ldots,i_j}|$.  Hence \eqref{lowbndspecnrm} holds.

\emph{2}.  The equality \eqref{specnrmtenprod} is a well known equality, which follows from 
\[(\cE\otimes\cT)\times((\otimes_{i=1}^k \x_i)\otimes(\otimes_{j=1}^d \y_j))=(\cE\times(\otimes_{i=1}^k \x_i))(\cT\times (\otimes_{j=1}^d \y_j)) \]
and from the definition of the spectral norm.

\emph{3}.  The inequality\eqref{subtenin} follows straightforwardly from \eqref{defspecnrm} by considering $\x_j$ that have support on $I_j$. 

\emph{4}.  The inequality \eqref{lowbndspecnrm} yields that 
\[\sum_{i_1\in[m_1],\ldots,i_d\in[m_d]} e_{i_1,\ldots,i_d}t_{i_1,\ldots,i_d}|x_{j_1,1}|\cdots |x_{j_d,d}|\le \|\cT\|_{\infty}\!\!\!\!
\sum_{i_1\in[m_1],\ldots,i_d\in[m_d]} e_{i_1,\ldots,i_d}|x_{j_1,1}|\cdots |x_{j_d,d}|\]
Apply now the characterization \eqref{infnrmnnc} to deduce the inequality \eqref{hadprodin}.

\emph{5}.  Clearly, if $\x_j\in\R^{m_j}_+$ for $j\in[d]$ then  $\cT(\bt)\times(\otimes_{j=1}^d \x_j)$ is a logconvex function for $\bt\in D$. Recall that the maximum of logconvex functions is a logconvex function.  The characterization \eqref{infnrmnnc} yields the logconvexity of $\|\cT(\bt)\|_{\infty}$.

\emph{6}.  Assume that $\cF=[f_{i_1,\ldots,i_d}],\cG=[g_{i_1,\ldots,i_d}]$, and $f_{i_1,\ldots,i_d}, g_{i_1,\ldots,i_d}>0$.  Define a logconvex map $\cT:\R^2\to \R^{\m}_+$
by the equality $\cT(s,t)=[f_{i_1,\ldots,i_d}^s g_{i_1,\ldots,i_d}^t]$.  Clearly $(\alpha,\beta)=\alpha(1,0)+\beta(0,1)$.  Hence the logconvexity of $\|\cT(s,t)\|_{\infty}$ yields
 \eqref{genkinin1}.  The general case in which some entries of $\cF$ or $\cG$ are zero follows from the continuity of the spectral norm.
 
 \emph{7}.  For $\cT\in\R^{\m}$ denote by
 \[\|\cT\|_{\ell_{\infty}}=\max\{|t_{i_1,\ldots,i_d}|, i_j\in[m_j], j\in[d]\},\]
 the $\ell_{\infty}$ norm of $\cT$ viewed as a vector.  
We shall use the inequalities:
 \begin{equation}\label{uplowestspecnrm}
 \|\cT\|_{\ell_{\infty}}\le \|\cT\|_{\infty}\le \|\cT\|\le \sqrt{M(d)}\|\cT\|_{\ell_{\infty}}, \quad M(d)=\prod_{j=1}^d m_j.
 \end{equation}
 Indeed, the first inequality is precisely~\eqref{lowbndspecnrm}.
The Cauchy-Schwarz inequality yield that $|\cT\times(\otimes_{j=1}^d\x_j)|\le \|\cT\| \|\otimes_{j=1}^d\x_j\|$.
Hence $\|\cT\|_{\infty}\le \|\cT\|$.   The last inequality in~\eqref{uplowestspecnrm} is straightforward. 

 Assume that $\cT\ge 0$.  Then
 \[\|\cT\|_{\ell_{\infty}}\le \|\cT^{\circ s}\|^{\frac{1}{s}}\le M(d)^{\frac{1}{2s}}\|\cT\|_{\ell_{\infty}}, \quad s>0.\]
 Let $s\to\infty$ to deduce \eqref{limspecnrm}.
 
 We now show that $\|\cT^{\circ s}\|^{\frac{1}{s}}_{\infty}$ is a decreasing function for $s\in (0,\infty)$.  It is enough to show that $\|\cT\|_{\infty}\ge \|\cT^{\circ t}\|_{\infty}^{\frac{1} {t}}$
 for $t>1$. Fix $t>1$.
   Let $\R^2\to \R_+^{\m}$ be the logconvex map $(a,b)\mapsto \cT(a,b)=\cT^{\circ(a+b)}$.
 Hence $\log \|\cT(a,b)\|_{\infty}$ is a convex function on $\R^2$.  Assume that $s>t$.  Note that
 \[(1,t-1)=\frac{t-1}{s-1}(1,s-1)+\frac{s-t}{s-1}(1,0).\]
Hence
\[\log\|\cT^{\circ t}\|_{\infty}\le \frac{(t-1)s}{(s-1)}\left(\frac{1}{s}\log\|\cT^{\circ s}\|_\infty\right)+\frac{s-t}{(s-1)}\log\|\cT\|_{\infty}.\]
Let $s\to\infty$ and use  \eqref{limspecnrm} to deduce 
\begin{equation}\label{specnrmtin}
\log\|\cT^{\circ t}\|_{\infty}\le (t-1)\log\|\cT\|_{\ell_{\infty}}+\log\|\cT\|_{\infty}
\end{equation}
Use \eqref{uplowestspecnrm} to deduce $t\log \|\cT\|_{\infty}\ge \log \|\cT^{\circ t}\|_{\infty}$.

\emph{8}. The inequality \eqref{increspropspectnrm} follows straightforwardly from \eqref{infnrmnnc}.  The first inequality of \eqref{ineqETtropnorm} follows from the 
inequality $\cT\circ \cE\le \|\cE\|_{\ell_{\infty}}\cT$ and \eqref{increspropspectnrm}.  The second inequality of \eqref{ineqETtropnorm} follow from the first inequality  of
\eqref{ineqETtropnorm} by letting $\cT=\pat \cE$.
\end{proof}

We conclude this section with the following remark.  In view of \eqref{limspecnrm} $\|\cT\|_{\ell_{\infty}}$ can be considered as a tropical version of $\|\cT\|_{\infty}$
for a nonnegative tensor.  Hence the inequalities \eqref{ineqETtropnorm} are analogs of the inequalities \eqref{rhotropin1} -- \eqref{rhotropin2}.

\renewcommand{\abstractname}{Acknowledgements}
\begin{abstract}
 Shmuel Friedland was partially supported by Simons collaboration grant for mathematicians.
\end{abstract}

\bibliographystyle{alpha}
\bibliography{tensor}

 \end{document}